\documentclass[twoside,english,american,a4paper]{scrartcl}
\usepackage[T1]{fontenc}
\usepackage[latin1]{inputenc}
\usepackage{babel}
\usepackage{prettyref}
\usepackage{amsthm}
\usepackage{amsmath}
\usepackage{amssymb}
\usepackage[unicode=true,pdfusetitle,
 bookmarks=true,bookmarksnumbered=false,bookmarksopen=false,
 breaklinks=false,pdfborder={0 0 1},backref=false,colorlinks=false]
 {hyperref}

\makeatletter
\theoremstyle{plain}
\newtheorem{thm}{\protect\theoremname}[section]
  \theoremstyle{definition}
  \newtheorem{defn}[thm]{\protect\definitionname}
  \theoremstyle{remark}
  \newtheorem{rem}[thm]{\protect\remarkname}
  \theoremstyle{plain}
  \newtheorem{prop}[thm]{\protect\propositionname}
  \theoremstyle{plain}
  \newtheorem{cor}[thm]{\protect\corollaryname}
  \theoremstyle{plain}
  \newtheorem{lem}[thm]{\protect\lemmaname}

\usepackage{helvet}
\usepackage[T1]{fontenc}

\usepackage{amsmath}
\usepackage{setspace}
\usepackage{amssymb}
\usepackage{enumerate}
\usepackage{url}
\usepackage{prettyref}

\newrefformat{prop}{Proposition \ref{#1}}
\newrefformat{lem}{Lemma \ref{#1}}
\newrefformat{thm}{Theorem \ref{#1}}
\newrefformat{cor}{Corollary \ref{#1}}
\newrefformat{rem}{Remark \ref{#1}}
\newrefformat{eq}{(\ref{#1})}
\newrefformat{item}{(\ref{#1})}

\makeatletter

\theoremstyle{definition}
\newtheorem*{hyp}{Hypotheses}

\usepackage[T1]{fontenc}    

\usepackage{a4wide}        
\usepackage{amsmath}

\usepackage{euscript}
\usepackage{mathtools}
\allowdisplaybreaks[1] 
\usepackage{amsfonts}
\newenvironment{keywords}{ \noindent\footnotesize\textbf{Keywords and phrases:}}{}

\newenvironment{class}{\noindent\footnotesize\textbf{Mathematics subject classification 2010:}}{}

\usepackage{color}

\newcommand*{\trace}{\operatorname{trace}}

\newcommand*{\dive}{\operatorname{div}}

\newcommand*{\Grad}{\operatorname{Grad}}
\newcommand*{\Div}{\operatorname{Div}}

\newcommand*{\grad}{\operatorname{grad}}

\newcommand*{\supp}{\operatorname{supp}}

\renewcommand*{\i}{\mathrm{i}}

\DeclareMathAccent{\Circ}{\mathalpha}{operators}{"17}

\renewcommand{\Im}{\operatorname{\mathfrak{Im}}}
\renewcommand{\Re}{\operatorname{\mathfrak{Re}}}

\renewcommand{\hat}{\widehat}

\renewcommand*{\epsilon}{\varepsilon}
\renewcommand*{\rho}{\varrho}

\makeatother

\usepackage{babel}

\AtBeginDocument{
  
}

\makeatother

  \addto\captionsamerican{\renewcommand{\corollaryname}{Corollary}}
  \addto\captionsamerican{\renewcommand{\definitionname}{Definition}}
  \addto\captionsamerican{\renewcommand{\lemmaname}{Lemma}}
  \addto\captionsamerican{\renewcommand{\propositionname}{Proposition}}
  \addto\captionsamerican{\renewcommand{\remarkname}{Remark}}
  \addto\captionsamerican{\renewcommand{\theoremname}{Theorem}}
  \addto\captionsenglish{\renewcommand{\corollaryname}{Corollary}}
  \addto\captionsenglish{\renewcommand{\definitionname}{Definition}}
  \addto\captionsenglish{\renewcommand{\lemmaname}{Lemma}}
  \addto\captionsenglish{\renewcommand{\propositionname}{Proposition}}
  \addto\captionsenglish{\renewcommand{\remarkname}{Remark}}
  \addto\captionsenglish{\renewcommand{\theoremname}{Theorem}}
  \providecommand{\corollaryname}{Corollary}
  \providecommand{\definitionname}{Definition}
  \providecommand{\lemmaname}{Lemma}
  \providecommand{\propositionname}{Proposition}
  \providecommand{\remarkname}{Remark}
\providecommand{\theoremname}{Theorem}

\begin{document}
\author{ Sascha Trostorff \\ Institut f\"ur Analysis, Fachrichtung Mathematik\\ Technische Universit\"at Dresden\\ Germany\\ sascha.trostorff@tu-dresden.de}

\title{On Integro-Differential Inclusions with Operator-valued Kernels.}

\maketitle
\begin{abstract} \textbf{Abstract}. We study integro-differential
inclusions in Hilbert spaces with operator-valued kernels and give
sufficient conditions for the well-posedness. We show that several
types of integro-differential equations and inclusions are covered
by the class of evolutionary inclusions and we therefore give criteria
for the well-posedness within this framework. As an example we apply
our results to the equations of visco-elasticity and to a class of
nonlinear integro-differential inclusions describing Phase Transition
phenomena in materials with memory.\end{abstract}

\begin{keywords} Integro-Differential Inclusion and Equations, Well-posedness
and Causality, Evolutionary Inclusions, Linear Material Laws, Maximal
Monotone Operators, Visco-Elasticity, Phase Transition Models.\end{keywords}

\begin{class} 35R09 (Integro-partial differential equations), 47G20
(Integro-Differential operators), 35F61 (Initial-boundary value problems
for nonlinear first-order systems), 46N20 (Applications to differential
and integral equations), 47J35 (Nonlinear evolution equations)  \end{class}

\newpage

\tableofcontents{} 

\newpage

\section{Introduction}

It appears that classical phenomena in mathematical physics, like
heat conduction, wave propagation or elasticity, show some memory
effects (see e.g. \cite{fabrizio1987mathematical,Nunziato1971}).
One way to mathematically model these effects is to use integro-differential
equations. Typically, linear integro-differential equations of hyperbolic
and parabolic type, discussed in the literature, are of the form 
\begin{equation}
\ddot{u}(t)+\intop_{-\infty}^{t}g(t-s)\ddot{u}(s)\mbox{ d}s+G^{\ast}Gu(t)-\intop_{-\infty}^{t}G^{\ast}h(t-s)Gu(s)\mbox{ d}s=f(t)\quad(t\in\mathbb{R}),\label{eq:hyper}
\end{equation}
and 
\begin{equation}
\dot{u}(t)+\intop_{-\infty}^{t}g(t-s)\dot{u}(s)\mbox{ d}s+G^{\ast}Gu(t)-\intop_{-\infty}^{t}G^{\ast}h(t-s)Gu(s)\mbox{ d}s=f(t)\quad(t\in\mathbb{R}),\label{eq:parabolic-1}
\end{equation}
respectively. In both cases $G$ denotes a closed, densely defined
linear operator on some Hilbert space, which, in applications, is
a differential operator with respect to the spatial variables. These
types of equations were treated by several authors, mostly assuming
that the kernels $g$ and $h$ are scalar-valued.\\
The theory of integro-differential equations has a long history and
there exists a large amount of works by several authors and we just
mention the monographs \cite{pruss1993evolutionary,Gripenberg1990}
and the references therein for possible approaches. Topics like well-posedness,
regularity and the asymptotic behaviour of solutions were studied
by several authors, even for semi-linear versions of \prettyref{eq:hyper}
or \prettyref{eq:parabolic-1} (e.g. \cite{Cavalcanti2003,Berrimi2006,Cannarsa2011}
for the hyperbolic and \cite{Aizicovici1980,Clement1992,Cannarsa2003}
for the parabolic case). \\
Our approach to deal with integro-differential equations invokes the
framework of evolutionary equations, introduced by Picard in \cite{Picard,Picard2010}.
The main idea is to rewrite the equations as problems of the form
\begin{equation}
\left(\partial_{0}M(\partial_{0}^{-1})+A\right)U=F,\label{eq:evol}
\end{equation}
or, more generally, as a differential inclusion. Here $\partial_{0}$
denotes the time-derivative established as a boundedly invertible
operator in a suitable exponentially weighted $L_{2}$-space. The
operator $M(\partial_{0}^{-1})$, called the linear material law,
is a bounded operator in time and space and is defined as an analytic,
operator-valued function of $\partial_{0}^{-1}$. The operator $A$
is assumed to be skew-selfadjoint. More generally, $A$ is a maximal
monotone, possibly set-valued, operator (we refer to the monographs
\cite{Brezis,hu2000handbook,showalter_book} for the theory of monotone
operators) and in this case \prettyref{eq:evol} becomes a differential
inclusion of the form 
\begin{equation}
(U,F)\in\partial_{0}M(\partial_{0}^{-1})+A,\label{eq:incl}
\end{equation}
which will be referred to as an evolutionary inclusion. These kinds
of problems were studied by the author in previous works (cf. \foreignlanguage{english}{\cite{Trostorff_2011,Trostorff2012_NA,Trostorff2012_nonlin_bd}})
and the well-posedness and causality was shown under suitable assumptions
on the material law $M(\partial_{0}^{-1})$. As it was already mentioned
in \cite{Kalauch2011}, the operator $M\left(\partial_{0}^{-1}\right)$
can be a convolution with an operator-valued function and we will
point out, which kind of linear material laws yield integro-differential
equations. We emphasize that problems of the form \prettyref{eq:hyper}
and \prettyref{eq:parabolic-1} are covered by \prettyref{eq:incl},
while the converse is false in general, unless additional compatibility
constraints on the operators $A$ and $M(\partial_{0}^{-1})$ are
imposed. Thus, on one hand our approach provides a unified solution
theory for a broader class of integro-differential inclusions. On
the other hand, since the solution theory for hyperbolic and parabolic
problems is the same, we cannot expect to obtain optimal regularity
results by this general approach, such as maximal regularity for parabolic-type
problems.\\
According to the solution theory for inclusions of the form \prettyref{eq:incl}
(see \cite[Theorem 3.7]{Trostorff2012_nonlin_bd} or \prettyref{thm:sol_theory}
in this article) it suffices to show the strict positive definiteness
of $\Re\partial_{0}M(\partial_{0}^{-1})$ in order to obtain well-posedness
of the problem. Besides existence, uniqueness and continuous dependence
we obtain the causality of the respective solution operators, which
enables us to treat initial value problems. \\
The article is structured as follows. In Section 2 we recall the notion
of linear material laws, evolutionary inclusions and we state the
solution theory for this class of problems. Section 3 is devoted to
the well-posedness of hyperbolic- and parabolic-type integro-differential
inclusions. We will state conditions for the involved kernels, which
imply the positive definiteness of $\Re\partial_{0}M(\partial_{0}^{-1})$
and therefore yield the well-posedness of the problems, where we focus
on kernels which are not differentiable. Furthermore, in Subsection
3.1 we will briefly discuss a way of how to treat initial value problems
(see \prettyref{rem:ivp}) as well as problems where the whole history
of the unknown is given (\prettyref{rem: history}) in the case of
a skew-selfadjoint operator $A$. Finally, we apply our findings in
Section 4 to concrete examples. At first, we consider the equations
of visco-elasticity, which were also treated by Dafermos \cite{Dafermos1970_asymp_stab,Dafermos1970_abtract_Volterra},
even for operator-valued kernels but under the stronger assumption
that the kernels are absolutely continuous. As a second example we
deal with a nonlinear integro-differential inclusion, arising in the
theory of phase transition problems in materials with long-term memory
effects (see \cite{Colli1993,Visintin1985}). \\
Throughout, every Hilbert space is assumed to be complex and the inner
product, denoted by $\langle\cdot|\cdot\rangle$ is linear in the
second and anti-linear in the first argument. Norms are usually denoted
by $|\cdot|$ except the operator-norm, which we denote by $\|\cdot\|.$

\section{Evolutionary inclusions}

In this section we recall the notion of evolutionary inclusions due
to \cite{Trostorff_2011,Trostorff2012_nonlin_bd}, based on the framework
of evolutionary equations introduced in \cite{Picard,Picard2010,Picard_McGhee}.
We begin to define the exponentially weighted $L_{2}$-space and the
time-derivative $\partial_{0}$, established as a normal, boundedly
invertible operator on this space. Using the spectral representation
of this time-derivative operator, we define so called linear material
laws as operator-valued $\mathcal{H}^{\infty}$-functions of $\partial_{0}^{-1}.$
In the second subsection we recall the solution theory for evolutionary
inclusions and the notion of causality.

\subsection{The time-derivative and linear material laws}

Throughout let $\nu\in\mathbb{R}$. As in \cite{Picard_McGhee,Picard,Kalauch2011}
we begin to introduce the exponentially weighted $L_{2}$-space%
\footnote{For convenience we always identify the equivalence classes with respect
to the equality almost everywhere with their respective representers.%
}.
\begin{defn}
We define the Hilbert space
\[
H_{\nu,0}(\mathbb{R})\coloneqq\left\{ f:\mathbb{R}\to\mathbb{C}\,\left|\, f\mbox{ measurable},\:\intop_{\mathbb{R}}|f(t)|^{2}e^{-2\nu t}\mbox{ d}t<\infty\right.\right\} 
\]
endowed with the inner-product 
\[
\langle f|g\rangle_{H_{\nu,0}}\coloneqq\intop_{\mathbb{R}}f(t)^{\ast}g(t)e^{-2\nu t}\mbox{ d}t\quad(f,g\in H_{\nu,0}(\mathbb{R})).
\]
\end{defn}
\begin{rem}
Obviously the operator 
\[
e^{-\nu m}:H_{\nu,0}(\mathbb{R})\to L_{2}(\mathbb{R})
\]
defined by $\left(e^{-\nu m}f\right)(t)=e^{-\nu t}f(t)$ for $t\in\mathbb{R}$
is unitary. 
\end{rem}
We define the derivative $\partial$ on $L_{2}(\mathbb{R})$ as the
closure of the operator 
\begin{align*}
\partial|_{C_{c}^{\infty}(\mathbb{R})}:C_{c}^{\infty}(\mathbb{R})\subseteq L_{2}(\mathbb{R}) & \to L_{2}(\mathbb{R})\\
\phi & \mapsto\phi',
\end{align*}
where $C_{c}^{\infty}(\mathbb{R})$ denotes the space of infinitely
differentiable functions on $\mathbb{R}$ with compact support. This
operator is known to be skew-selfadjoint (see \cite[p. 198, Example 3]{Yosida})
and its spectral representation is given by the Fourier transform
$\mathcal{F}$, which is given as the unitary continuation of 
\[
\left(\mathcal{F}\phi\right)(t)\coloneqq\frac{1}{\sqrt{2\pi}}\intop_{\mathbb{R}}e^{-\i st}\phi(s)\mbox{ d}s\quad(t\in\mathbb{R})
\]
for functions $\phi\in L_{1}(\mathbb{R})\cap L_{2}(\mathbb{R}),$
i.e., we have 
\begin{equation}
\partial=\mathcal{F}^{\ast}(\i m)\mathcal{F},\label{eq:Fourier}
\end{equation}
where $m:D(m)\subseteq L_{2}(\mathbb{R})\to L_{2}(\mathbb{R})$ denotes
the multiplication-by-the-argument operator ($\left(mf\right)(t)=tf(t)$)
with maximal domain $D(m).$ 
\begin{defn}
We define the operator $\partial_{\nu}$ on $H_{\nu,0}(\mathbb{R})$
by 
\[
\partial_{\nu}\coloneqq\left(e^{-\nu m}\right)^{-1}\partial e^{-\nu m}
\]
and obtain again a skew-selfadjoint operator. From \prettyref{eq:Fourier}
we immediately get 
\[
\partial_{\nu}=\left(e^{-\nu m}\right)^{-1}\mathcal{F}^{\ast}\i m\mathcal{F}e^{-\nu m},
\]
which yields the spectral representation for $\partial_{\nu}$ by
the so-called \emph{Fourier-Laplace transform }$\mathcal{L}_{\nu}\coloneqq\mathcal{F}e^{-\nu m}:H_{\nu,0}(\mathbb{R})\to L_{2}(\mathbb{R})$. 
\end{defn}
An easy computation shows, that for $\phi\in C_{c}^{\infty}(\mathbb{R})$
we get $\phi'=\partial_{\nu}\phi+\nu\phi,$ which leads to the following
definition.
\begin{defn}
We define the operator $\partial_{0,\nu}\coloneqq\partial_{\nu}+\nu$,
the \emph{time-derivative} on $H_{\nu,0}(\mathbb{R})$. If the choice
of $\nu\in\mathbb{R}$ is clear from the context we will write $\partial_{0}$
instead of $\partial_{0,\nu}.$\end{defn}
\begin{rem}
Another way to introduce $\partial_{0,\nu}$ is by taking the closure
of the usual derivative of test-functions with respect to the topology
in $H_{\nu,0}(\mathbb{R}),$ i.e. 
\[
\partial_{0,\nu}=\overline{\partial|_{C_{c}^{\infty}(\mathbb{R})}}^{H_{\nu,0}(\mathbb{R})\oplus H_{\nu,0}(\mathbb{R})}.
\]

\end{rem}
We state some properties of the derivative $\partial_{0,\nu}$ and
refer to \cite{Kalauch2011,Picard_McGhee} for the proofs.
\begin{prop}
Let $\nu>0$. Then the following statements hold:

\begin{enumerate}[(a)]

\item  The operator $\partial_{0,\nu}$ is normal and $0\in\rho(\partial_{0,\nu})$
with $\|\partial_{0,\nu}^{-1}\|=\frac{1}{\nu}.$

\item  $\partial_{0,\nu}=\mathcal{L}_{\nu}^{\ast}(\i m+\nu)\mathcal{L}_{\nu}.$

\item  For $u\in H_{\nu,0}(\mathbb{R})$ we have \foreignlanguage{english}{\textup{$\left(\partial_{0,\nu}^{-1}u\right)(t)=\intop_{-\infty}^{t}u(s)\mbox{ d}s$}}
for almost every $t\in\mathbb{R}.$%
\footnote{This shows, that for $\nu>0$ the operator $\partial_{0,\nu}^{-1}$
is causal, while for $\nu<0$ we get the anti-causal operator given
by $\partial_{0,\nu}^{-1}u=-\intop_{\cdot}^{\infty}u(s)\mbox{ d}s$
(see \cite{Kalauch2011}).%
} 

\end{enumerate}
\end{prop}
Of course, the operator $\partial_{0,\nu}$ can be lifted in the canonical
way to Hilbert space-valued functions and for convenience we will
use the same notation for the derivative on scalar-valued and on Hilbert
space-valued functions. The space of Hilbert space-valued functions,
which are square-integrable with respect to the exponentially weighted
Lebesgue measure will be denoted by $H_{\nu,0}(\mathbb{R};H)$ for
$\nu\in\mathbb{R}.$ Using the spectral representation for the inverse
time-derivative $\partial_{0,\nu}^{-1}$ for $\nu>0$, we introduce
linear material laws as follows.
\begin{defn}
For $r>0$ let $M:B_{\mathbb{C}}(r,r)\to L(H)$ be a bounded, analytic
function. Then we define the bounded linear operator 
\[
M\left(\frac{1}{\i m+\nu}\right):L_{2}(\mathbb{R};H)\to L_{2}(\mathbb{R};H)
\]
for $\nu>\frac{1}{2r}$ by $\left(M\left(\frac{1}{\i m+\nu}\right)f\right)(t)=M\left(\frac{1}{\i t+\nu}\right)f(t)$
for $t\in\mathbb{R}$ and the \emph{linear material law} $M(\partial_{0,\nu}^{-1})$
by 
\[
M(\partial_{0,\nu}^{-1})\coloneqq\mathcal{L}_{\nu}^{\ast}M\left(\frac{1}{\i m+\nu}\right)\mathcal{L}_{\nu}\in L(H_{\nu,0}(\mathbb{R};H),H_{\nu,0}(\mathbb{R};H)).
\]

\end{defn}
Note that the operator $M(\partial_{0,\nu}^{-1})$, as a function
of $\partial_{0,\nu}^{-1}$, commutes with the derivative $\partial_{0,\nu},$
in the sense that $\partial_{0,\nu}M(\partial_{0,\nu}^{-1})\supseteq M(\partial_{0,\nu}^{-1})\partial_{0,\nu}.$
\begin{rem}
The assumed analyticity of the mapping $M$ is needed to ensure the
\emph{causality} (see \prettyref{thm:sol_theory}) of the operator
$M(\partial_{0,\nu}^{-1})$ using a Paley-Wiener-type result (cf.
\cite[Theorem 19.2]{rudin1987real}). 
\end{rem}

\subsection{Well-posedness and causality of evolutionary inclusions}

In \cite{Trostorff2012_nonlin_bd} the following type of a differential
inclusion was considered:
\begin{equation}
(U,F)\in\partial_{0,\nu}M(\partial_{0,\nu}^{-1})+A_{\nu}\label{eq:evol_eq}
\end{equation}
where $A\subseteq H\oplus H$ is a maximal monotone relation%
\footnote{Frequently, maximal monotone operators are defined as set- or multi-valued
mappings, i.e. $A:D(A)\subseteq H\to\mathcal{P}(H)$. However, we
identify this mapping with the relation given by 
\[
\left\{ (x,y)\in H\oplus H\,|\, y\in A(x)\right\} ,
\]
and denote this relation again by $A$.%
} satisfying $(0,0)\in A$, $A_{\nu}$ is its extension to $H_{\nu,0}(\mathbb{R};H)\oplus H_{\nu,0}(\mathbb{R};H)$
given by%
\footnote{If the choice of $\nu$ is clear from the context, we may omit the
index $\nu$ for the extension and simply write $A$.%
} 
\[
A_{\nu}\coloneqq\left\{ (u,v)\in H_{\nu,0}(\mathbb{R};H)^{2}\,|\,(u(t),v(t))\in A\;(t\in\mathbb{R}\:\mbox{a.e.})\right\} ,
\]
which again defines a maximal monotone relation (see \cite[Exemple 2.3.3]{Brezis}),
$F\in H_{\nu,0}(\mathbb{R};H)$ is an arbitrary source term and $U\in H_{\nu,0}(\mathbb{R};H)$
is the unknown. For this class of problems the following solution
theory was established.
\begin{thm}[{\cite[Theorem 3.7]{Trostorff2012_nonlin_bd}}]
 \label{thm:sol_theory}Let $A\subseteq H\oplus H$ be maximal monotone
with $(0,0)\in A$ and let $M:B_{\mathbb{C}}(r,r)\to L(H)$ be analytic,
bounded and such that there exists $c>0$ such that for all $z\in B_{\mathbb{C}}(r,r)$
the following holds 
\begin{equation}
\Re z^{-1}M(z)\geq c.\label{eq:solv}
\end{equation}
Then for each $\nu>\frac{1}{2r}$ the inverse relation $\left(\partial_{0,\nu}M(\partial_{0,\nu}^{-1})+A_{\nu}\right)^{-1}$
defines a Lipschitz-continuous mapping on $H_{\nu,0}(\mathbb{R};H).$
Moreover, the inverse is causal, i.e. 
\[
\chi_{\mathbb{R}_{\leq a}}(m)\overline{\left(\partial_{0,\nu}M(\partial_{0,\nu}^{-1})+A_{\nu}\right)^{-1}}\chi_{\mathbb{R}_{\leq a}}(m)=\chi_{\mathbb{R}_{\leq a}}(m)\overline{\left(\partial_{0,\nu}M(\partial_{0,\nu}^{-1})+A_{\nu}\right)^{-1}}
\]
for each $a\in\mathbb{R}.$%
\footnote{Here we denote by $\chi_{\mathbb{R}_{\leq a}}(m)$ the cut-off operator
given by $\left(\chi_{\mathbb{R}_{\leq a}}(m)f\right)(t)=\chi_{\mathbb{R}_{\leq a}}(t)f(t).$%
}
\end{thm}
This means that under the hypotheses of \prettyref{thm:sol_theory},
Problem \prettyref{eq:evol_eq} is well-posed, i.e. the uniqueness,
existence and continuous dependence on the data $F$ of a solution
$U$ is guaranteed. However, \prettyref{eq:evol_eq} just holds in
the sense of 
\[
(U,F)\in\overline{\partial_{0,\nu}M(\partial_{0,\nu}^{-1})+A_{\nu}}
\]
where the closure is taken with respect to the topology on $H_{\nu,0}(\mathbb{R};H).$
If $A$ is a maximal monotone, linear operator, we can avoid the closure,
by using the concept of extrapolation spaces, so-called Sobolev-chains
with respect to the operator $A+1$ and $\partial_{0,\nu}$ (see \cite{Picard2000},
\cite[Chapter 2]{Picard_McGhee}). In this context Equation \prettyref{eq:evol_eq}
holds in the space $H_{\nu,-1}(\mathbb{R};H_{-1}(A+1))$, where we
denote by $\left(H_{\nu,k}(\mathbb{R})\right)_{k\in\mathbb{Z}}$ the
Sobolev-chain associated to $\partial_{0,\nu}.$ Using that $M(\partial_{0,\nu}^{-1})$
and $A$ commute with $\partial_{0,\nu}$, one derives the following
corollary from \prettyref{thm:sol_theory}.
\begin{cor}[{\cite[Theorem 6.2.5]{Picard_McGhee}}]
\label{cor:sol_extrapolation}Let $A:D(A)\subseteq H\to H$ a linear
maximal monotone operator and $M:B_{\mathbb{C}}(r,r)\to L(H)$ as
in \prettyref{thm:sol_theory}. Then the solution operator $\left(\partial_{0,\nu}M(\partial_{0,\nu}^{-1})+A_{\nu}\right)^{-1}$
extends to a bounded linear operator on $H_{\nu,k}(\mathbb{R};H)$
for each $k\in\mathbb{Z}.$ 
\end{cor}

\section{Integro-differential inclusions}

In this section we introduce an abstract type of integro-differential
inclusions with operator-valued kernels, which covers hyperbolic-
and parabolic-type problems. This abstract type allows to treat convolutions
with the unknown as well as with the derivatives (with respect to
time and space) of the unknown. We introduce the space $L_{1,\mu}(\mathbb{R}_{\geq0};L(H))$
for $\mu\in\mathbb{R}$ as the space of weakly measurable functions
$B:\mathbb{R}_{\geq0}\to L(H)$ (i.e. for every $x,y\in H$ the function
$t\mapsto\langle B(t)x|y\rangle$ is measurable) such that the function
$t\mapsto\|B(t)\|$ is measurable%
\footnote{If $H$ is separable, then the weak measurability implies the measurability
of $t\mapsto\|B(t)\|.$%
} and 
\[
|B|_{L_{1,\mu}(\mathbb{R}_{\ge0};L(H))}\coloneqq\intop_{0}^{\infty}e^{-\mu t}\|B(t)\|\mbox{ d}t<\infty.
\]
Note that $L_{1,\mu}(\mathbb{R}_{\geq0};L(H))\hookrightarrow L_{1,\nu}(\mathbb{R}_{\geq0};L(H))$
for $\mu\leq\nu.$ For a function $B\in L_{1,\mu}(\mathbb{R}_{\geq0};L(H))$
we can establish the Fourier-transform of $B$ as a function on the
lower half-plane $[\mathbb{R}]-\i[\mathbb{R}_{\geq\mu}]\coloneqq\{t-\i\nu\,|\, t\in\mathbb{R},\nu\geq\mu\}$
by defining 
\[
\langle\hat{B}(t-\i\nu)x|y\rangle\coloneqq\frac{1}{\sqrt{2\pi}}\intop_{0}^{\infty}e^{-\i ts}e^{-\nu s}\langle B(s)x|y\rangle\mbox{ d}s\quad(t\in\mathbb{R},\nu\geq\mu)
\]

for $x,y\in H.$ Obviously the function $t-\i\nu\mapsto\hat{B}(t-\i\nu)$
is bounded on $[\mathbb{R}]-\i[\mathbb{R}_{\geq\mu}]$ with values
in the bounded operators on $H$ and satisfies $|\hat{B}|_{L_{\infty}([\mathbb{R}]-\i[\mathbb{R}_{>\mu}];L(H))}\leq\frac{1}{\sqrt{2\pi}}|B|_{L_{1,\mu}(\mathbb{R}_{\geq0};L(H))}.$
Moreover it is analytic on the open half plane $[\mathbb{R}]-\i[\mathbb{R}_{>\mu}]$%
\footnote{Note that scalar analyticity on a norming set and local boundedness
is equivalent to analyticity (see \cite[Theorem 3.10.1]{hille1957functional}). %
}. For $B\in L_{1,\mu}(\mathbb{R}_{\geq0};L(H))$ we define the convolution
operator as follows.
\begin{lem}
\label{lem:conv_op-1} Let \textbf{$B\in L_{1,\mu}(\mathbb{R}_{\geq0};L(H))$
}for some $\mu\in\mathbb{R}$. We denote by $S(\mathbb{R};H)$ the
space of simple functions on $\mathbb{R}$ with values in $H$. Then
for each $\nu\geq\mu$ the \emph{convolution operator}%
\footnote{The integral is defined in the weak sense.%
} 
\begin{align*}
B\ast:S(\mathbb{R};H)\subseteq H_{\nu,0}(\mathbb{R};H) & \to H_{\nu,0}(\mathbb{R};H)\\
u & \mapsto\left(t\mapsto\intop_{\mathbb{R}}B(t-s)u(s)\,\mathrm{d}s\right)
\end{align*}
is bounded with $\|B\ast\|_{L(H_{\nu,0}(\mathbb{R};H))}\leq|B|_{L_{1,\nu}(\mathbb{R}_{\geq0};L(H))}$.
Hence, it can be extended to a bounded linear operator on $H_{\nu,0}(\mathbb{R};H).$\end{lem}
\begin{proof}
Let $\nu\geq\mu.$ Then we estimate for $u\in S(\mathbb{R};H)$ using
Young's inequality 
\begin{align*}
\intop_{\mathbb{R}}e^{-2\nu t}|(B\ast u)(t)|^{2}\mbox{ d}t & \leq\intop_{\mathbb{R}}\left(\intop_{\mathbb{R}}e^{-\nu(t-s)}\|B(t-s)\|e^{-\nu s}|u(s)|\mbox{ d}s\right)^{2}\mbox{ d}t\\
 & \leq\left(\intop_{\mathbb{R}}e^{-\nu t}\|B(t)\|\mbox{ d}t\right)^{2}\intop_{\mathbb{R}}|u(t)|^{2}e^{-2\nu t}\mbox{ d}t,
\end{align*}
which yields $B\ast u\in H_{\nu,0}(\mathbb{R};H)$ and 
\[
|B\ast u|_{H_{\nu,0}(\mathbb{R};H)}\leq\intop_{0}^{\infty}e^{-\nu t}\|B(t)\|\mbox{ d}t\,|u|_{H_{\nu,0}(\mathbb{R};H)}.
\]
This completes the proof.\end{proof}
\begin{rem}
\label{rem:conv_sobolev}Note that since $B\ast$ commutes with $\partial_{0,\nu}$
we can extend $B\ast$ to a bounded linear operator on $H_{\nu,k}(\mathbb{R};H)$
for each $k\in\mathbb{Z}.$\end{rem}
\begin{cor}
\label{cor:Neumann}Let $B\in L_{1,\mu}(\mathbb{R}_{\geq0};L(H))$.
Then $\lim_{\nu\to\infty}|B|_{L_{1,\nu}(\mathbb{R}_{\geq0};L(H))}=0$
and thus, $\lim_{\nu\to\infty}\|B\ast\|_{L(H_{\nu,0}(\mathbb{R};H))}=0.$ \end{cor}
\begin{proof}
This is an immediate consequence of \prettyref{lem:conv_op-1} and
the monotone convergence theorem.\end{proof}
\begin{lem}
Let $B\in L_{1,\mu}(\mathbb{R};L(H))$ for some $\mu\geq0$ and $u\in H_{\nu,0}(\mathbb{R};H)$
for $\nu\geq\mu.$ Then 
\[
\left(\mathcal{L}_{\nu}(B\ast u)\right)(t)=\sqrt{2\pi}\hat{B}(t-\i\nu)\left(\mathcal{L}_{\nu}u(t)\right)
\]
for almost every $t\in\mathbb{R}.$\end{lem}
\begin{proof}
The proof is a classical computation using Fubini's Theorem and we
omit it.
\end{proof}

\subsection{Material laws for hyperbolic-type problems}

In this subsection we consider material laws of the form 
\begin{equation}
M_{0}(z)=1+\sqrt{2\pi}\hat{C}(-\i z^{-1})\quad(z\in B_{\mathbb{C}}(r,r))\label{eq:material_law_hyper}
\end{equation}
and 
\begin{equation}
M_{1}(z)=\left(1-\sqrt{2\pi}\hat{B}(-\i z^{-1})\right)^{-1}\quad(z\in B_{\mathbb{C}}(r,r))\label{eq:material_law_hyper_resolvent}
\end{equation}
for some $C\in L_{1,\mu}(\mathbb{R}_{\geq0};L(H_{0})),B\in L_{1,\mu}(\mathbb{R}_{\geq0};L(H_{1}))$,
where $\mu\geq0$. According to \prettyref{cor:Neumann} there exists
$\mu_{0}>\mu,$ such that $\|B\ast\|_{L(H_{\nu,0}(\mathbb{R};H_{1}))}<1$
for each $\nu\geq\mu_{0}$. To ensure that the function $M_{1}$ defines
a linear material law, we choose $r\coloneqq\frac{1}{2\mu_{0}}.$\\
A typical example for a material law for hyperbolic-type systems is
of the form 
\[
M(z)=\left(\begin{array}{cc}
M_{0}(z) & 0\\
0 & M_{1}(z)
\end{array}\right).
\]
Then the corresponding integro differential inclusion \prettyref{eq:evol_eq}
reads as 
\begin{equation}
\left(\left(\begin{array}{c}
v\\
q
\end{array}\right),\left(\begin{array}{c}
f\\
g
\end{array}\right)\right)\in\partial_{0}\left(\begin{array}{cc}
1+C\ast & 0\\
0 & (1-B\ast)^{-1}
\end{array}\right)+A,\label{eq:hyper-1}
\end{equation}
where $U=\left(\begin{array}{c}
v\\
q
\end{array}\right),F=\left(\begin{array}{c}
f\\
g
\end{array}\right)\in H_{\nu,0}(\mathbb{R};H_{0}\oplus H_{1}).$ If $A$ is given as a block operator matrix of the form 
\[
\left(\begin{array}{cc}
0 & G^{\ast}\\
-G & 0
\end{array}\right)
\]
for some closed linear operator $G:D(G)\subseteq H_{0}\to H_{1}$,
and if $g=0$, we obtain by the second row of \prettyref{eq:hyper-1}
\[
\partial_{0}(1-B\ast)^{-1}q=Gv
\]
or, equivalently, 
\[
q=\partial_{0}^{-1}(1-B\ast)Gv.
\]
If we plug this representation of $q$ into the first line of \prettyref{eq:hyper-1}
we get that 
\[
\partial_{0}(1+C\ast)v+G^{\ast}\partial_{0}^{-1}(1-B\ast)Gv=f
\]
which gives, by defining $u\coloneqq\partial_{0}^{-1}v$ 
\[
\partial_{0}^{2}(1+C\ast)u+A^{\ast}(1-B\ast)Au=f.
\]
Thus, a solution theory for \prettyref{eq:hyper-1} will cover hyperbolic
equations of the form \prettyref{eq:hyper}. A semi-linear version
of this equation was treated in \cite{Cannarsa2011} for scalar-valued
kernels, where criteria for the well-posedness and the exponential
stability were given. Also in \cite{Pruss2009} this type of equation
was treated for scalar-valued kernels and besides well-posedness,
the polynomial stability was addressed. In both works the well-posedness
(the existence and uniqueness of mild solutions) was shown under certain
conditions on the kernel by techniques developed for evolutionary
integral equations (see \cite{pruss1993evolutionary}). We will show
that the assumptions on the kernels made in both articles can be weakened
such that the well-posedness of the problem can still be shown, even
for operator-valued kernels. \\
In order to show the solvability condition \prettyref{eq:solv} for
the material laws $M_{0}$ and $M_{1}$ we have to show that there
exist $r_{1},c>0$ with $r_{1}\leq r$ such that for all $z\in B_{\mathbb{C}}(r_{1},r_{1}):$
\begin{equation}
\Re z^{-1}(1+\sqrt{2\pi}\hat{C}(-\i z^{-1}))\geq c\label{eq:solv_C}
\end{equation}
and 
\begin{equation}
\Re z^{-1}(1-\sqrt{2\pi}\hat{B}(-\i z^{-1}))^{-1}\geq c.\label{eq:solv_B}
\end{equation}

\begin{rem}
\label{rem:absolute_cont}One standard assumption for scalar-valued
kernels is absolute continuity. In our case this means that there
exists a function $B'\in L_{1,\mu}(\mathbb{R}_{\geq0};L(H_{1}))$
for some $\mu\geq0$ such that 
\[
B(t)=\intop_{0}^{t}B'(s)\mbox{ d}s+B(0)\quad(t\in\mathbb{R}_{\geq0})
\]
for the kernel $B.$\\
Note that due to the absolute continuity, \textbf{$B$ }is an element
of \foreignlanguage{english}{$L_{1,\nu}(\mathbb{R}_{\geq0};L(H_{1}))$}\textbf{
}for each $\nu>\mu$ and we choose $\nu$ large enough, such that
$|B|_{L_{1,\nu}(\mathbb{R}_{\geq0};L(H_{1}))}\!<\!1.$ In this case
\prettyref{eq:solv_B} can be easily verified. Using the Neumann series
we obtain 
\begin{align*}
z^{-1}(1-\sqrt{2\pi}\hat{B}(-\i z^{-1}))^{-1} & =z^{-1}\sum_{k=0}^{\infty}\left(\sqrt{2\pi}\hat{B}(-\i z^{-1})\right)^{k}\\
 & =z^{-1}+z^{-1}\sqrt{2\pi}\hat{B}(-\i z^{-1})\sum_{k=0}^{\infty}\left(\sqrt{2\pi}\hat{B}(-\i z^{-1})\right)^{k}
\end{align*}
for $z\in B_{\mathbb{C}}\left(\frac{1}{2\nu},\frac{1}{2\nu}\right).$
The Fourier transform of $B$ can be computed by 
\[
\hat{B}(-\i z^{-1})=z\left(\hat{B'}(-\i z^{-1})+\frac{1}{\sqrt{2\pi}}B(0)\right)
\]
and hence, we can estimate 
\begin{align*}
 & \Re z^{-1}(1-\sqrt{2\pi}\hat{B}(-\i z^{-1}))^{-1}\\
 & =\Re z^{-1}+\Re\left(\sqrt{2\pi}\hat{B'}(-\i z^{-1})+B(0)\right)\sum_{k=0}^{\infty}\left(\sqrt{2\pi}\hat{B}(-\i z^{-1})\right)^{k}\\
 & \geq\nu-\frac{|\sqrt{2\pi}\hat{B'}(-\i\left(\cdot\right)^{-1})|_{L_{\infty}\left(B_{\mathbb{C}}\left(\frac{1}{2\nu},\frac{1}{2\nu}\right);L(H_{1})\right)}+\|B(0)\|_{L(H_{1})}}{1-|\sqrt{2\pi}\hat{B}(-\i\left(\cdot\right)^{-1})|_{L_{\infty}\left(B_{\mathbb{C}}\left(\frac{1}{2\nu},\frac{1}{2\nu}\right);L(H_{1})\right)}}\\
 & \geq\nu-\frac{|B'|_{L_{1,\nu}(\mathbb{R}_{\geq0};L(H_{1}))}+\|B(0)\|_{L(H_{1})}}{1-|B|_{L_{1,\nu}(\mathbb{R}_{\geq0};L(H_{1}))}}
\end{align*}
for every $z\in B_{\mathbb{C}}\left(\frac{1}{2\nu},\frac{1}{2\nu}\right).$
Since $\frac{|B'|_{L_{1,\nu}(\mathbb{R}_{\geq0};L(H_{1}))}+\|B(0)\|_{L(H_{1})}}{1-|B|_{L_{1,\nu}(\mathbb{R}_{\geq0};L(H_{1}))}}\to\|B(0)\|_{L(H_{1})}$
as $\nu\to\infty$, this yields the assertion. An analogous result
holds for absolutely continuous $C$.
\end{rem}
In the case when $C$ and $B$ are not assumed to be differentiable
in a suitable sense, the conditions \prettyref{eq:solv_B} and \prettyref{eq:solv_C}
are hard to verify. We now state some hypotheses for $B$ and $C$
and show in the remaining part of this subsection, that these conditions
imply \prettyref{eq:solv_B} and \prettyref{eq:solv_C}.

\begin{hyp} Let $T\in L_{1,\mu}(\mathbb{R}_{\geq0};L(G)),$ where
$G$ is an arbitrary Hilbert space and $\mu\geq0.$ Then $T$ satisfies
the hypotheses \prettyref{eq:selfadjoint},\prettyref{eq:commute}
and \prettyref{eq:imaginary} respectively, if

\begin{enumerate}[(i)]

\item \label{eq:selfadjoint}for all $t\in\mathbb{R}_{\geq0}$ the
operator $T(t)$ is selfadjoint,

\item \label{eq:commute} for all $s,t\in\mathbb{R}_{\geq0}$ the
operators $T(t)$ and $T(s)$ commute,

\item \label{eq:imaginary} there exists $d\geq0,$ $\nu_{0}\geq\mu$
such that for all $t\in\mathbb{R}$
\[
t\Im\hat{T}(t-\i\nu_{0})\leq d.
\]

\end{enumerate}

\end{hyp}
\begin{rem}
$\,$\begin{enumerate}[(a)]

\item  If $T$ satisfies the hypothesis \prettyref{eq:selfadjoint},
then 
\[
\Im\hat{T}(t-\i\nu_{0})=\frac{1}{\sqrt{2\pi}}\intop_{0}^{\infty}\sin(-ts)e^{-\nu_{0}s}T(s)\mbox{ d}s=-\Im\hat{T}(-t-\i\nu_{0})\quad(t\in\mathbb{R})
\]
and thus \prettyref{eq:imaginary} holds if and only if 
\[
t\Im\hat{T}(t-\i\nu_{0})\leq d\quad(t\in\mathbb{R}_{>0}).
\]
\item Note that in \cite{Pruss2009} and \cite{Cannarsa2011} the
kernel is assumed to be real-valued. Thus, \prettyref{eq:selfadjoint}
and \prettyref{eq:commute} are trivially satisfied. In \cite{Pruss2009}
we find the assumption, that the kernel should be non-increasing and
non-negative, i.e., $T(s)\geq0$ and $T(t)-T(s)\leq0$ for each $t\geq s\geq0.$
Note that these assumptions imply 
\[
\langle\left(e^{-\nu t}T(t)-e^{-\nu s}T(s)\right)x|x\rangle=e^{-\nu t}\langle\left(T(t)-T(s)\right)x|x\rangle+(e^{-\nu t}-e^{-\nu s})\langle T(s)x|x\rangle\leq0
\]
for every $t\geq s\geq0,\nu\geq0$ and $x\in G$. Hence, we estimate
for $t>0$ and $x\in G$ 
\begin{align*}
 & \langle\Im\hat{T}(t-\i\nu_{0})x|x\rangle\\
 & =\frac{1}{\sqrt{2\pi}}\intop_{0}^{\infty}\sin(-ts)e^{-\nu_{0}s}\langle T(s)x|x\rangle\mbox{ d}s\\
 & =\frac{1}{\sqrt{2\pi}}\sum_{k=0}^{\infty}\left(\intop_{2k\frac{\pi}{t}}^{(2k+1)\frac{\pi}{t}}\sin(-ts)e^{-\nu_{0}s}\langle T(s)x|x\rangle\mbox{ d}s+\intop_{(2k+1)\frac{\pi}{t}}^{2(k+1)\frac{\pi}{t}}\sin(-ts)e^{-\nu_{0}s}\langle T(s)x|x\rangle\mbox{ d}s\right)\\
 & =\frac{1}{\sqrt{2\pi}}\sum_{k=0}^{\infty}\left(\intop_{2k\frac{\pi}{t}}^{(2k+1)\frac{\pi}{t}}\sin(-ts)e^{-\nu_{0}s}\langle T(s)u|u\rangle\mbox{ d}s\right.+\\
 & \left.\quad+\intop_{2k\frac{\pi}{t}}^{(2k+1)\frac{\pi}{t}}\sin(-ts-\pi)e^{-\nu_{0}\left(s+\frac{\pi}{t}\right)}\left\langle \left.T\left(s+\frac{\pi}{t}\right)u\right|u\right\rangle \mbox{ d}s\right)\\
 & =\frac{1}{\sqrt{2\pi}}\sum_{k=0}^{\infty}\intop_{2k\frac{\pi}{t}}^{(2k+1)\frac{\pi}{t}}\sin(-ts)\left\langle \left.\left(e^{-\nu_{0}s}T(s)-e^{-\nu_{0}\left(s+\frac{\pi}{t}\right)}T\left(s+\frac{\pi}{t}\right)\right)u\right|u\right\rangle \mbox{ d}s\\
 & =\frac{1}{\sqrt{2\pi}}\sum_{k=0}^{\infty}\intop_{2k\frac{\pi}{t}}^{(2k+1)\frac{\pi}{t}}\sin(ts)\left\langle \left.\left(e^{-\nu_{0}\left(s+\frac{\pi}{t}\right)}T\left(s+\frac{\pi}{t}\right)-e^{-\nu_{0}s}T(s)\right)u\right|u\right\rangle \mbox{ d}s\leq0,
\end{align*}
which yields \prettyref{eq:imaginary} for $d=0$ according to (a).
The authors of \cite{Cannarsa2011} assume that the integrated kernel
defines a positive definite convolution operator on $L_{2}(\mathbb{R}_{\geq0}).$
However, according to \cite[Proposition 2.2 (a)]{Cannarsa2011}, this
condition also implies \prettyref{eq:imaginary} for $d=0$.

\item By hypothesis \prettyref{eq:imaginary} we also cover kernels
of bounded variation. Indeed, if $T\in L_{1}(\mathbb{R}_{\geq0};L(G))$
is a function of strong bounded variation (see \cite[Definition 3.2.4]{hille1957functional}),
then $\sup_{t\in\mathbb{R}}\|t\hat{T}(t-\i\nu_{0})\|<\infty.$ This
is a stronger estimate than \prettyref{eq:imaginary}, which suggests
that our results apply to a broader class than functions of strong
bounded variation. However, to the authors best knowledge, there exists
no characterization of functions satisfying an estimate of the form
\prettyref{eq:imaginary} even in the scalar-valued case.

\end{enumerate}
\end{rem}
We now prove that a kernel $T$ satisfying the hypotheses \prettyref{eq:selfadjoint}
and \prettyref{eq:imaginary}, also satisfies an estimate of the form
\prettyref{eq:imaginary} for every $\nu\geq\nu_{0}.$ 
\begin{lem}
\label{lem:one_nu_every_nu}Assume that $T\in L_{1,\mu}(\mathbb{R}_{\geq0};L(G))$
satisfies the hypotheses \prettyref{eq:selfadjoint} and \prettyref{eq:imaginary}.
Then we have for all $\nu\geq\nu_{0}$ and $t\in\mathbb{R}$
\[
t\Im\hat{T}(t-\i\nu)\leq4d.
\]
\end{lem}
\begin{proof}
Let $x\in G$ and $\nu\geq\mu.$ We define the function 
\[
f(t)\coloneqq\langle T(t)x|x\rangle\quad(t\in\mathbb{R})
\]
which is real-valued, due to the selfadjointness of $T(t)$ and we
estimate 
\[
\intop_{\mathbb{R}}|f(t)|e^{-\mu t}\mbox{ d}t\leq\intop_{\mathbb{R}}\|T(t)\|e^{-\mu t}\mbox{ d}t\,|x|^{2}
\]
which shows $f\in L_{1,\mu}(\mathbb{R}_{>0})$. We observe that 
\begin{align*}
\left\langle \hat{T}(t-\i\nu)x|x\right\rangle  & =\left\langle \left.\frac{1}{\sqrt{2\pi}}\intop_{\mathbb{R}}e^{-\i ts}e^{-\nu s}T(s)\mbox{ d}s\: x\,\right|x\right\rangle \\
 & =\frac{1}{\sqrt{2\pi}}\intop_{\mathbb{R}}e^{\i ts}e^{-\nu s}\langle T(s)x|x\rangle\mbox{ d}s\\
 & =\hat{f}(-t-\i\nu)
\end{align*}
for each $t\in\mathbb{R},\nu\geq\mu.$ Hence, by 
\[
\langle\Im\hat{T}(t-\i\nu)x|x\rangle=-\Im\langle\hat{T}(t-\i\nu)x|x\rangle=-\Im\hat{f}(-t-\i\nu)=\Im f(t-\i\nu),
\]
it suffices to prove $t\Im\hat{f}(t-\i\nu)$ is bounded from above
for $\nu\geq\nu_{0},\; t\in\mathbb{R}$ under the condition that $t\Im\hat{f}(t-\i\nu_{0})\leq d$
for each $t\in\mathbb{R}.$ For this purpose we follow the strategy
in \cite[Lemma 3.4]{Cannarsa2003} and employ the Poisson formula
for the half plain (see \cite[p. 149]{stein2003fourier}) in order
to compute the values of the harmonic function $\Im\hat{f}(\cdot):[\mathbb{R}]-\i\left[\mathbb{R}_{\geq\mu}\right]\to\mathbb{R}.$
This gives, using $\Im\hat{f}(-s-\i\nu)=-\Im\hat{f}(s-\i\nu)$ 
\begin{align*}
\Im\hat{f}(t-\i\nu) & =\frac{1}{\pi}\intop_{-\infty}^{\infty}\frac{\nu-\nu_{0}}{(t-s)^{2}+(\nu-\nu_{0})^{2}}\Im\hat{f}(s-\i\nu_{0})\mbox{ d}s\\
 & =\frac{\nu-\nu_{0}}{\pi}\left(\intop_{0}^{\infty}\frac{1}{(t+s)^{2}+(\nu-\nu_{0})^{2}}\Im\hat{f}(-s-\i\nu_{0})\mbox{ d}s\right.+\\
 & \quad+\left.\intop_{0}^{\infty}\frac{1}{(t-s)^{2}+(\nu-\nu_{0})^{2}}\Im\hat{f}(s-\i\nu_{0})\mbox{ d}s\right)\\
 & =\frac{\nu-\nu_{0}}{\pi}\intop_{0}^{\infty}\left(\frac{1}{(t-s)^{2}+(\nu-\nu_{0})^{2}}-\frac{1}{(t+s)^{2}+(\nu-\nu_{0})^{2}}\right)\Im\hat{f}(s-\i\nu_{0})\mbox{ d}s\\
 & =4t\frac{\nu-\nu_{0}}{\pi}\intop_{0}^{\infty}\left(\frac{s}{\left((t-s)^{2}+(\nu-\nu_{0})^{2}\right)\left((t+s)^{2}+(\nu-\nu_{0})^{2}\right)}\right)\Im\hat{f}(s-\i\nu_{0})\mbox{ d}s,
\end{align*}
which implies 
\begin{align*}
 & \quad t\Im\hat{f}(t-\i\nu)\\
 & =4t^{2}\frac{\nu-\nu_{0}}{\pi}\intop_{0}^{\infty}\left(\frac{s}{\left((t-s)^{2}+(\nu-\nu_{0})^{2}\right)\left((t+s)^{2}+(\nu-\nu_{0})^{2}\right)}\right)\Im\hat{f}(s-\i\nu_{0})\mbox{ d}s\\
 & \leq4d\frac{\nu-\nu_{0}}{\pi}\intop_{0}^{\infty}\frac{1}{\left(t-s\right)^{2}+\left(\nu-\nu_{0}\right)^{2}}\mbox{ d}s\\
 & =4d\frac{1}{\pi}\intop_{-\infty}^{\frac{t}{\nu-\nu_{0}}}\frac{1}{r^{2}+1}\mbox{ d}r\\
 & \leq4d.\tag*{\qedhere}
\end{align*}
\end{proof}
\begin{rem}
The fact that $T$ satisfies \prettyref{eq:imaginary} not only for
one parameter $\nu_{0}>0$ but for all $\nu\geq\nu_{0}$ is crucial
for the causality of the solution operator. Indeed, if one is only
interested in the well-posedness of problems of the form \prettyref{eq:incl}
for one particular parameter $\nu_{0}$ it would be sufficient to
assume hypothesis \prettyref{eq:imaginary} and omit assumption \prettyref{eq:selfadjoint}.\\
We are now able to state our main results of this subsection.\end{rem}
\begin{prop}
\label{prop:pos_C} Let $C$ satisfy the hypotheses \prettyref{eq:selfadjoint}
and \prettyref{eq:imaginary}. Then there exists $0<r_{1}\leq r$
and $c>0$ such that for all $z\in B_{\mathbb{C}}(r_{1},r_{1})$ condition
\prettyref{eq:solv_C} is satisfied.\end{prop}
\begin{proof}
Let $z\in B_{\mathbb{C}}(r_{1},r_{1})$, where we will choose $r_{1}$
later on. Using the representation $z^{-1}=\i t+\nu$ for $t\in\mathbb{R}$
and $\nu>\frac{1}{2r_{1}}$ we estimate 
\begin{align*}
\Re\left(\i t+\nu\right)\left(1+\sqrt{2\pi}\hat{C}(t-\i\nu)\right) & =\nu\left(1+\sqrt{2\pi}\Re\hat{C}(t-\i\nu)\right)-\sqrt{2\pi}t\Im\hat{C}(t-\i\nu)\\
 & \geq\nu\left(1-|C|_{L_{1,\nu}(\mathbb{R}_{\geq0};L(H))}\right)-4\sqrt{2\pi}d,
\end{align*}
 where we have used \prettyref{lem:one_nu_every_nu}. Using \prettyref{cor:Neumann},
this yields the assertion.\end{proof}
\begin{prop}
\label{prop:pos_B}Let $B$ satisfy the hypotheses \prettyref{eq:selfadjoint}-\prettyref{eq:imaginary}.
Then there exists $0<r_{1}\leq r$ and $c>0$ such that for all $z\in B_{\mathbb{C}}(r_{1},r_{1})$
the condition \prettyref{eq:solv_B} is satisfied.\end{prop}
\begin{proof}
Let $x\in H_{1}$ and choose $r_{1}<\min\left\{ \frac{1}{2\nu_{0}},r,\frac{1-|B|_{L_{1,\nu_{0}}(\mathbb{R}_{>0};L(H_{1}))}}{8d\sqrt{2\pi}}\right\} .$
Let $z\in B_{\mathbb{C}}(r_{1},r_{1})$. Then $z^{-1}=\i t+\nu$ for
some $t\in\mathbb{R},\nu>\frac{1}{2r_{1}}$. Since the operator $1-\sqrt{2\pi}\hat{B}(t-\i\nu)$
is bounded and boundedly invertible, so is its adjoint, which is given
by $1-\sqrt{2\pi}\hat{B}(-t-\i\nu)$ since $B(s)$ is selfadjoint
for each $s\in\mathbb{R}.$ We compute 
\[
\Re\langle(\i t+\nu)(1-\sqrt{2\pi}\hat{B}(t-\i\nu))^{-1}x|x\rangle=\Re\langle(\i t+\nu)(|1-\sqrt{2\pi}\hat{B}(t-\i\nu)|^{2})^{-1}(1-\sqrt{2\pi}\hat{B}(-t-\i\nu))x|x\rangle.
\]
We define the operator $D\coloneqq|1-\sqrt{2\pi}\hat{B}(t-\i\nu)|{}^{-1}$.
Furthermore, note that due to \prettyref{eq:commute}, we have that
the operators $\hat{B}(\cdot)$ commute. This especially implies,
that $\hat{B}(t-\i\nu)$ is normal and hence $D$ and $1-\sqrt{2\pi}\hat{B}(-t-\i\nu)$
commute. Thus, we can estimate the real part by 
\begin{align*}
 & \quad\Re\langle(\i t+\nu)D^{2}(1-\sqrt{2\pi}\hat{B}(-t-\i\nu))x|x\rangle\\
 & =\Re(-\i t+\nu)\langle(1-\sqrt{2\pi}\hat{B}(-t-\i\nu))Dx|Dx\rangle\\
 & =\nu\langle\left(1-\sqrt{2\pi}\Re\hat{B}(-t-\i\nu)\right)Dx|Dx\rangle+t\langle\sqrt{2\pi}\Im\hat{B}(-t-\i\nu)Dx|Dx\rangle\\
 & \geq\nu\left(1-\|\sqrt{2\pi}\hat{B}(-t-\i\nu)\|\right)|Dx|^{2}-\sqrt{2\pi}\langle t\Im\hat{B}(t-\i\nu)Dx|Dx\rangle\\
 & \geq\left(\nu(1-|B|_{L_{1,\nu}(\mathbb{R}_{\geq0};L(H_{1}))})-\sqrt{2\pi}4d\right)|Dx|^{2}\\
 & \geq\left(\nu(1-|B|_{L_{1,\nu_{0}}(\mathbb{R}_{\geq0};L(H_{1}))})-\sqrt{2\pi}4d\right)|Dx|^{2},
\end{align*}
where we have used \prettyref{lem:one_nu_every_nu}. Using now the
inequality 
\[
|x|=|D^{-1}Dx|=\left|\left(1-\sqrt{2\pi}\hat{B}(t-\i\nu)\right)Dx\right|\leq(1+|B|_{L_{1,\nu_{0}}(\mathbb{R}_{\geq0};L(H_{1}))})|Dx|
\]
we arrive at 
\begin{align*}
\Re\langle(\i t+\nu)D^{2}(1-\sqrt{2\pi}\hat{B}(-t-\i\nu))x|x\rangle & \geq\frac{\nu(1-|B|_{L_{1,\nu_{0}}(\mathbb{R}_{\geq0};L(H_{1}))})-\sqrt{2\pi}4d}{\left(1+|B|_{L_{1,\nu_{0}}(\mathbb{R}_{\geq0};L(H_{1}))}\right)^{2}}|x|^{2}
\end{align*}
which shows the assertion since $\nu>\frac{1}{2r_{1}}$.
\end{proof}
In applications it turns out that \prettyref{eq:hyper-1} is just
assumed to hold for positive times, i.e. on $\mathbb{R}_{>0}$ and
the equation is completed by initial conditions. So for instance,
we can require that the unknowns $v$ and $q$ are supported on the
positive real line and attain some given initial values at time $0.$
Then we arrive at a usual initial value problem. Since, due to the
convolution with $B$ and $C$ the history of $v$ and $q$ has an
influence on the equation for positive times, we can, instead of requiring
an initial value at $0$, prescribe the values of $v$ and $q$ on
the whole negative real-line. This is a standard problem in delay-equations
and it is usually treated by introducing so-called history-spaces
(see e.g. \cite{hale1993introduction,diekmann1995delay}). However,
following the idea of \cite{Kalauch2011} we can treat this kind of
equations as a problem of the form \prettyref{eq:hyper-1} with a
modified right-hand side. Since we have to invoke the theory of Sobolev-chains
in order to deal with such problems, we assume that $A:D(A)\subseteq H\to H$
is linear and maximal monotone. Let us treat the case of classical
initial value problems first. 
\begin{rem}[Initial value problem]
 \label{rem:ivp}For $\left(f,g\right)\in H_{\nu,0}(\mathbb{R};H)$
with $\supp f,\supp g\subseteq\mathbb{R}_{\geq0}$ we consider the
differential equation \foreignlanguage{english}{
\[
\left(\partial_{0}\left(\begin{array}{cc}
1+C\ast & 0\\
0 & (1-B\ast)^{-1}
\end{array}\right)+A\right)\left(\begin{array}{c}
v\\
q
\end{array}\right)=\left(\begin{array}{c}
f\\
g
\end{array}\right)
\]
}on $\mathbb{R}_{>0}$ completed by initial conditions of the form
\[
v(0+)=v^{(0)}\mbox{ and }q(0+)=q^{(0)},
\]
where we assume $(v^{(0)},q^{(0)})\in D(A)$. We assume that the solvability
conditions \prettyref{eq:solv_B} and \prettyref{eq:solv_C} are fulfilled.
Assume that a pair $(v,q)\in\chi_{\mathbb{R}_{>0}}(m_{0})[H_{\nu,1}(\mathbb{R};H)]$%
\footnote{This means that we find a pair $(w,p)\in H_{\nu,1}(\mathbb{R};H)$
that coincides with $(v,q)$ for positive times.%
} solves this problem. Then we get 
\[
\partial_{0}\left(\begin{array}{cc}
1+C\ast & 0\\
0 & (1-B\ast)^{-1}
\end{array}\right)\left(\begin{array}{c}
v-\chi_{\mathbb{R}_{>0}}\otimes v^{(0)}\\
q-\chi_{\mathbb{R}_{>0}}\otimes q^{(0)}
\end{array}\right)+A\left(\begin{array}{c}
v\\
q
\end{array}\right)=\left(\begin{array}{c}
f\\
g
\end{array}\right)
\]
on $\mathbb{R},$ which is equivalent to 
\begin{equation}
\left(\partial_{0}\left(\begin{array}{cc}
1+C\ast & 0\\
0 & (1-B\ast)^{-1}
\end{array}\right)+A\right)\left(\begin{array}{c}
v\\
q
\end{array}\right)=\left(\begin{array}{c}
f\\
g
\end{array}\right)+\left(\begin{array}{cc}
1+C\ast & 0\\
0 & (1-B\ast)^{-1}
\end{array}\right)\delta\otimes\left(\begin{array}{c}
v^{(0)}\\
q^{(0)}
\end{array}\right).\label{eq:ivp}
\end{equation}
We note that $\delta\in H_{\nu,-1}(\mathbb{R})$ and according to
\prettyref{rem:conv_sobolev}, the operators $1+C\ast$ and $(1-B\ast)^{-1}$
have a continuous extension to $H_{\nu,-1}(\mathbb{R};H).$ We claim
that \prettyref{eq:ivp} is the proper formulation of the initial
value problem in our framework. According to \prettyref{cor:sol_extrapolation}
this equation admits a unique solution $(v,q)\in H_{\nu,-1}(\mathbb{R};H)$
and due to the causality of the solution operator we get $\supp v,\supp q\subseteq\mathbb{R}_{\geq0}$.
We derive from \prettyref{eq:ivp} that
\[
\left(\partial_{0}\left(\begin{array}{cc}
1+C\ast & 0\\
0 & (1-B\ast)^{-1}
\end{array}\right)+A\right)\left(\begin{array}{c}
v-\chi_{\mathbb{R}_{>0}}\otimes v^{(0)}\\
q-\chi_{\mathbb{R}_{>0}}\otimes q^{(0)}
\end{array}\right)=\left(\begin{array}{c}
f\\
g
\end{array}\right)-\chi_{\mathbb{R}_{>0}}\otimes A\left(\begin{array}{c}
v^{(0)}\\
q^{(0)}
\end{array}\right)
\]
which gives $\left(\begin{array}{c}
v-\chi_{\mathbb{R}_{>0}}\otimes v^{(0)}\\
q-\chi_{\mathbb{R}_{>0}}\otimes q^{(0)}
\end{array}\right)\in H_{\nu,0}(\mathbb{R};H).$ However, we also get that 
\begin{align*}
 & \partial_{0}\left(\begin{array}{cc}
1+C\ast & 0\\
0 & (1-B\ast)^{-1}
\end{array}\right)\left(\begin{array}{c}
v-\chi_{\mathbb{R}_{>0}}\otimes v^{(0)}\\
q-\chi_{\mathbb{R}_{>0}}\otimes q^{(0)}
\end{array}\right)\\
= & \left(\begin{array}{c}
f\\
g
\end{array}\right)-A\left(\begin{array}{c}
v\\
q
\end{array}\right)\in H_{\nu,0}(\mathbb{R};H_{-1}(A+\i)),
\end{align*}
and hence, 
\[
\left(\begin{array}{c}
v-\chi_{\mathbb{R}_{>0}}\otimes v^{(0)}\\
q-\chi_{\mathbb{R}_{>0}}\otimes q^{(0)}
\end{array}\right)\in H_{\nu,1}(\mathbb{R};H_{-1}(A+1)).
\]
Using the Sobolev-embedding Theorem (see \cite[Lemma 3.1.59]{Picard_McGhee}
or \cite[Lemma 5.2]{Kalauch2011}) we obtain that \foreignlanguage{english}{$\left(\begin{array}{c}
v-\chi_{\mathbb{R}_{>0}}\otimes v^{(0)}\\
q-\chi_{\mathbb{R}_{>0}}\otimes q^{(0)}
\end{array}\right)$} is continuous with values in $H_{-1}(A+1)$ and hence, due to causality,
\[
0=\left(\begin{array}{c}
v-\chi_{\mathbb{R}_{>0}}\otimes v^{(0)}\\
q-\chi_{\mathbb{R}_{>0}}\otimes q^{(0)}
\end{array}\right)(0-)=\left(\begin{array}{c}
v-\chi_{\mathbb{R}_{>0}}\otimes v^{(0)}\\
q-\chi_{\mathbb{R}_{>0}}\otimes q^{(0)}
\end{array}\right)(0+)
\]
in $H_{-1}(A+1)$ and thus 
\[
\left(\begin{array}{c}
v\\
q
\end{array}\right)(0+)=\left(\begin{array}{c}
v^{(0)}\\
q^{(0)}
\end{array}\right)\mbox{ in }H_{-1}(A+1).
\]

\end{rem}
$\,$
\begin{rem}[Problems with prescribed history]
\label{rem: history} For $\left(f,g\right)\in H_{\nu,0}(\mathbb{R};H)$
with $\supp f,\supp g\subseteq\mathbb{R}_{\geq0}$ we again consider
the equation 
\begin{equation}
\left(\partial_{0}\left(\begin{array}{cc}
1+C\ast & 0\\
0 & (1-B\ast)^{-1}
\end{array}\right)+A\right)\left(\begin{array}{c}
v\\
q
\end{array}\right)=\left(\begin{array}{c}
f\\
g
\end{array}\right)\label{eq:hyper_again}
\end{equation}
on $\mathbb{R}_{>0}$ and the initial conditions 
\[
v|_{\mathbb{R}_{<0}}=v_{(-\infty)},\: v(0+)=v_{(-\infty)}(0-)\mbox{ and }q|_{\mathbb{R}_{<0}}=q_{(-\infty)},\: q(0+)=q_{(-\infty)}(0-).
\]
We assume that $v_{(-\infty)}\in H_{\nu,0}(\mathbb{R};H_{0})$ with
$\supp v_{(-\infty)}\subseteq\mathbb{R}_{\leq0}$ and $\left(1+C\ast\right)v_{(-\infty)}\in\chi_{\mathbb{R}_{>0}}(m_{0})[H_{\nu,1}(\mathbb{R};H_{0})]$
as well as $q{}_{(-\infty)}\in H_{\nu,0}(\mathbb{R};H_{1})$ with
$\supp q_{(-\infty)}\subseteq\mathbb{R}_{\leq0}$, and $\left(1-B\ast\right)^{-1}q_{(-\infty)}\in\chi_{\mathbb{R}_{>0}}(m_{0})[H_{\nu,1}(\mathbb{R};H_{1})]$.
Moreover, we assume $\left(v_{(-\infty)}(0-),q_{(-\infty)}(0-)\right)\in D(A)$.
We want to determine an evolutionary equation for $w\coloneqq\chi_{\mathbb{R}_{>0}}v$
and $p\coloneqq\chi_{\mathbb{R}_{>0}}q.$ We have that 
\begin{align}
\left(\begin{array}{c}
f\\
g
\end{array}\right) & =\chi_{\mathbb{R}_{>0}}\left(\partial_{0}\left(\begin{array}{cc}
1+C\ast & 0\\
0 & (1-B\ast)^{-1}
\end{array}\right)+A\right)\left(\begin{array}{c}
v\\
q
\end{array}\right)\nonumber \\
 & =\chi_{\mathbb{R}_{>0}}\left(\partial_{0}\left(\begin{array}{cc}
1+C\ast & 0\\
0 & (1-B\ast)^{-1}
\end{array}\right)+A\right)\left(\begin{array}{c}
w+v_{(-\infty)}\\
p+q_{(-\infty)}
\end{array}\right)\nonumber \\
 & =\chi_{\mathbb{R}_{>0}}\left(\partial_{0}\left(\begin{array}{cc}
1+C\ast & 0\\
0 & (1-B\ast)^{-1}
\end{array}\right)+A\right)\left(\begin{array}{c}
w\\
p
\end{array}\right)\nonumber \\
 & \quad+\chi_{\mathbb{R}_{>0}}\partial_{0}\mbox{\ensuremath{\left(\begin{array}{c}
(1+C\ast)v{}_{(-\infty)}\\
(1-B\ast)^{-1}q_{(-\infty)}
\end{array}\right)}}.\label{eq:comp}
\end{align}
Hence, we arrive at the following equation for $(w,p):$ 
\begin{align*}
\chi_{\mathbb{R}_{>0}}\left(\partial_{0}\left(\begin{array}{cc}
1+C\ast & 0\\
0 & (1-B\ast)^{-1}
\end{array}\right)+A\right)\left(\begin{array}{c}
w\\
p
\end{array}\right) & =\left(\begin{array}{c}
f\\
g
\end{array}\right)-\chi_{\mathbb{R}_{>0}}\partial_{0}\mbox{\ensuremath{\left(\begin{array}{c}
(1+C\ast)v{}_{(-\infty)}\\
(1-B\ast)^{-1}q_{(-\infty)}
\end{array}\right)}}.
\end{align*}
Note that we can omit the cut-off function on the left hand side due
to the causality of the operators. The conditions $v(0+)=v_{(-\infty)}(0-)$
and $q(0+)=q_{(-\infty)}(0-)$ are now classical initial conditions
for the unknowns $w$ and $p$. Hence, following \prettyref{rem:ivp}
we end up with the following evolutionary equation for $(w,p)$:
\begin{align}
\left(\partial_{0}\left(\begin{array}{cc}
1+C\ast & 0\\
0 & (1-B\ast)^{-1}
\end{array}\right)+A\right)\left(\begin{array}{c}
w\\
p
\end{array}\right) & =\left(\begin{array}{c}
f\\
g
\end{array}\right)-\chi_{\mathbb{R}_{>0}}\partial_{0}\mbox{\ensuremath{\left(\begin{array}{c}
(1+C\ast)v{}_{(-\infty)}\\
(1-B\ast)^{-1}q_{(-\infty)}
\end{array}\right)}}+\nonumber \\
 & +\left(\begin{array}{cc}
1+C\ast & 0\\
0 & (1-B\ast)^{-1}
\end{array}\right)\delta\otimes\left(\begin{array}{c}
v_{(-\infty)}(0-)\\
q_{(-\infty)}(0-)
\end{array}\right).\label{eq:pre_history}
\end{align}
This equation possesses a unique solution in $H_{\nu,-1}(\mathbb{R};H)$
with $\supp w,\supp p\subseteq\mathbb{R}_{>0}$ due to the causality
of the solution operator. Like in \prettyref{rem:ivp} we get that
\[
\left(\begin{array}{c}
w-\chi_{\mathbb{R}_{>0}}\otimes v_{(-\infty)}(0-)\\
p-\chi_{\mathbb{R}_{>0}}\otimes q_{(-\infty)}(0-)
\end{array}\right)\in H_{\nu,0}(\mathbb{R};H)
\]
from which we derive, using Equation \prettyref{eq:pre_history},
that 
\[
\left(\begin{array}{c}
w(0+)\\
p(0+)
\end{array}\right)=\left(\begin{array}{c}
v_{(-\infty)}(0-)\\
q_{(-\infty)}(0-)
\end{array}\right)
\]
in $H_{-1}(A+1)$. We are now able to define the original solution
by setting 
\[
v\coloneqq w+v_{(-\infty)}\in H_{\nu,0}(\mathbb{R};H_{0})\mbox{ and }q=p+q_{(-\infty)}\in H_{\nu,0}(\mathbb{R};H_{1}).
\]
Indeed, $v$ and $q$ satisfy the initial conditions by definition
and on $\mathbb{R}_{>0}$ the solution $(v,q)$ satisfies the differential
equation \prettyref{eq:hyper_again} according to the computation
done in \prettyref{eq:comp}.
\end{rem}

\subsection{Material laws for parabolic-type problems}

For parabolic-type problems it turns out that the material law is
typically given as a combination of terms of the form $M_{0}$ or
$M_{1}$, which were discussed in the previous section and terms of
the form 
\[
M_{2}(z)=z\left(1+\sqrt{2\pi}\hat{C}(-\i z^{-1})\right)\quad(z\in B_{\mathbb{C}}(r,r))
\]
and 
\[
M_{3}(z)=z\left(1-\sqrt{2\pi}\hat{B}(-\i z^{-1})\right)^{-1}\quad(z\in B_{\mathbb{C}}(r,r)),
\]
where $B\in L_{1,\mu}(\mathbb{R}_{\geq0};L(H_{1}))$ and $C\in L_{1,\mu}(\mathbb{R}_{\geq0};L(H_{0}))$
for some $\mu\geq0$ and $r>0$ is small enough, such that $M_{3}$
is a linear material law.\\
Frequently we find material laws given by 
\[
M(z)=\left(\begin{array}{cc}
M_{0}(z) & 0\\
0 & M_{3}(z)
\end{array}\right).
\]
The corresponding integro-differential inclusion is then given by
\begin{equation}
\left(\left(\begin{array}{c}
u\\
q
\end{array}\right),\left(\begin{array}{c}
f\\
g
\end{array}\right)\right)\in\partial_{0}\left(\begin{array}{cc}
1+C\ast & 0\\
0 & 0
\end{array}\right)+\left(\begin{array}{cc}
0 & 0\\
0 & \left(1-B\ast\right)^{-1}
\end{array}\right)+A.\label{eq:parabolic}
\end{equation}
Again, in the special case, when $g=0$ and $A$ is of the form $\left(\begin{array}{cc}
0 & G^{\ast}\\
-G & 0
\end{array}\right)$ for some closed densely defined linear operator $G:D(G)\subseteq H_{0}\to H_{1}$
this yields the parabolic equation 
\begin{equation}
\partial_{0}u+C\ast\partial_{0}u+G^{\ast}Gu-G^{\ast}\left(B\ast Gu\right)=f,\label{eq:parabolic-2}
\end{equation}
showing that problems of the form \prettyref{eq:parabolic-1} are
covered by \prettyref{eq:parabolic}. Such problems were considered
in \cite{Cannarsa2003} for scalar-valued kernels, where besides the
well-posedness the asymptotic behaviour was addressed. As it turns
out the solution theory for this kind of problem is quite easy in
comparison to the solution theory for the hyperbolic case.
\begin{prop}
\label{prop:pos_def_parabolic}There exist $0<r_{1}\leq r$ and $c>0$
such that the material laws $M_{2}$ and $M_{3}$ satisfy the solvability
condition \prettyref{eq:solv}.\end{prop}
\begin{proof}
Using \prettyref{cor:Neumann} we estimate for $M_{2}$ 
\[
\Re(1+\sqrt{2\pi}\hat{C}(-\i z^{-1}))\geq1-\sup_{z\in B_{\mathbb{C}}(r,r)}\|\sqrt{2\pi}\hat{C}(-\i z^{-1})\|\to1\quad(r\to0+).
\]
For $M_{3}$ we use the Neumann-series and estimate 
\begin{align*}
\Re\left(1-\sqrt{2\pi}\hat{B}(-\i z^{-1})\right)^{-1} & =1+\Re\sqrt{2\pi}\hat{B}(-\i z^{-1})\sum_{k=0}^{\infty}\left(\sqrt{2\pi}\hat{B}(-\i z^{-1})\right)^{k}\\
 & \geq1-\frac{\sup_{z\in B_{\mathbb{C}}(r,r)}\|\sqrt{2\pi}\hat{B}(-\i z^{-1})\|_{L(H_{1})}}{1-\sup_{z\in B_{\mathbb{C}}(r,r)}\|\sqrt{2\pi}\hat{B}(-\i z^{-1})\|_{L(H_{1})}}\\
 & \to1\quad(r\to0+).\tag*{\qedhere}
\end{align*}
\end{proof}
\begin{rem}
\label{rem:vlasov} In \cite{Vlasov2010} the following kind of a
parabolic-type integro-differential equation was considered: 
\begin{equation}
\partial_{0}u+C\ast\partial_{0}u+Lu-B\ast Lu=f,\label{eq:parabolic_vlasov}
\end{equation}
where $L:D(L)\subseteq H_{0}\to H_{0}$ is a selfadjoint strictly
positive definite operator and $B,C\in L_{1,\mu}(\mathbb{R}_{\geq0};L(H_{0}))$.
It is remarkable that no further assumptions on the kernel $C$ are
imposed, although \prettyref{eq:parabolic_vlasov} seems to be of
the form \prettyref{eq:parabolic-2}, where $C$ would have to verify
the hypotheses \prettyref{eq:selfadjoint} and \prettyref{eq:imaginary}.
The reason for that is that \prettyref{eq:parabolic_vlasov} can be
rewritten as a parabolic system, which is not of the classical form
\prettyref{eq:parabolic}. Indeed, instead of \prettyref{eq:parabolic_vlasov}
we consider 
\[
\partial_{0}u+(1+C\ast)^{-1}(1-B\ast)Lu=(1+C\ast)^{-1}f.
\]
In this case the well-posedness can be shown, without imposing additional
hypotheses on $C.$ First we write the problem in the form given in
\prettyref{eq:evol}. For doing so, consider the operator $L:D(L)\subseteq H_{1}(\sqrt{L})\to H_{0}.$
Recall that $H_{1}(\sqrt{L})$ is the domain of $\sqrt{L}$ equipped
with the inner product $(u,v)\mapsto\langle\sqrt{L}u|\sqrt{L}v\rangle.$
We compute the adjoint of this operator. First observe that for $g\in H_{1}(\sqrt{L})$
we get 
\begin{align*}
\langle g|Lf\rangle_{H_{0}} & =\langle\sqrt{L}g|\sqrt{L}f\rangle_{H_{0}}\\
 & =\langle g|f\rangle_{H_{1}(\sqrt{L})}
\end{align*}
for each $f\in D(L)$ and thus $g\in D(L^{\ast}).$ Furthermore, if
$g\in D(L^{\ast})$ there exists $h\in H_{1}(\sqrt{L})$ such that
for all $f\in D(L)$
\[
\langle g|Lf\rangle_{H_{0}}=\langle h|f\rangle_{H_{1}(\sqrt{L})}=\langle\sqrt{L}h|\sqrt{L}f\rangle_{H_{0}}=\langle h|Lf\rangle_{H_{0}}.
\]
Since $L$ has dense range we conclude that $g=h\in H_{1}(\sqrt{L}).$
Thus, the adjoint is given by the identity $1:H_{1}(\sqrt{L})\subseteq H_{0}\to H_{1}(\sqrt{L}).$
We rewrite Equation \prettyref{eq:parabolic_vlasov} in the following
way 
\[
\left(\partial_{0}\left(\begin{array}{cc}
1 & 0\\
0 & 0
\end{array}\right)+\left(\begin{array}{cc}
0 & 0\\
0 & (1-B\ast)^{-1}(1+C\ast)
\end{array}\right)+\left(\begin{array}{cc}
0 & 1\\
-L & 0
\end{array}\right)\right)\left(\begin{array}{c}
u\\
q
\end{array}\right)=\left(\begin{array}{c}
(1+C\ast)^{-1}f\\
0
\end{array}\right).
\]
Note that this is now an equation in the space $H_{\nu,0}(\mathbb{R};H_{1}(\sqrt{L})\oplus H_{0}).$
The strict positive definiteness of $\Re(1-B\ast)^{-1}(1+C\ast)$
follows from the strict positive definiteness of $\Re(1-B\ast)^{-1}$
(see \prettyref{prop:pos_def_parabolic}) and the fact that $\|(1-B\ast)^{-1}C\ast\|_{L(H_{\nu,0}(\mathbb{R};H_{0}))}\to0$
as $\nu\to\infty.$
\end{rem}

\section{Examples}

In this section we apply our findings of Subsection 3.1 to two concrete
examples. In the first example we deal with the equations of visco-elasticity,
while in the second example a class integro-differential inclusions
arising in the theory of phase transition with superheating and supercooling
effects (see \cite{Colli1993}) is addressed.

\subsection{Visco-Elasticity}

Throughout let $\Omega\subseteq\mathbb{R}^{3}$ be an arbitrary domain.
We begin by defining the Hilbert spaces and operators involved.
\begin{defn}
We consider the space 
\[
L_{2}(\Omega)^{3\times3}\coloneqq\left\{ \Psi=(\Psi_{ij})_{i,j\in\{1,2,3\}}\,|\,\forall i,j\in\{1,2,3\}:\Psi_{ij}\in L_{2}(\Omega)\right\} 
\]
equipped with the inner product 
\[
\langle\Psi|\Phi\rangle\coloneqq\intop_{\Omega}\trace(\Psi(x)^{\ast}\Phi(x))\mbox{ d}x\quad(\Psi,\Phi\in L_{2}(\Omega)^{3\times3}).
\]
It is obvious that $L_{2}(\Omega)^{3\times3}$ becomes a Hilbert space
and that 
\[
L_{2,\mathrm{sym}}(\Omega)\coloneqq\left\{ \Psi\in L_{2}(\Omega)^{3\times3}\,|\,\Psi(x)^{T}=\Psi(x)\quad(x\in\Omega\mbox{ a.e.})\right\} 
\]
defines a closed subspace of $L_{2}(\Omega)^{3\times3}$ and therefore
$L_{2,\mathrm{sym}}(\Omega)$ is also a Hilbert space. We introduce
the operator 
\begin{align*}
\Grad|_{C_{c}^{\infty}(\Omega)^{3}}:C_{c}^{\infty}(\Omega)^{3}\subseteq L_{2}(\Omega)^{3} & \to L_{2,\mathrm{sym}}(\Omega)\\
(\phi_{i})_{i\in\{1,2,3\}} & \mapsto\left(\frac{1}{2}(\partial_{i}\phi_{j}+\partial_{j}\phi_{i})\right)_{i,j\in\{1,2,3\}},
\end{align*}
which turns out to be closable and we denote its closure by $\Grad_{c}.$
Moreover we define $\Div\coloneqq-\Grad_{c}^{\ast}.$
\end{defn}
For $\Phi\in C_{c}^{1}(\Omega)^{3\times3}\cap L_{2,\mathrm{sym}}(\Omega)$
one can compute $\Div\Phi$ by 
\[
\left(\Div\Phi\right)_{i\in\{1,2,3\}}=\left(\sum_{j=1}^{3}\partial_{j}\Phi_{ij}\right)_{i\in\{1,2,3\}}.
\]
The equations of linear elasticity in a domain $\Omega\subseteq\mathbb{R}^{3}$
read as follows (see e.g. \cite[p. 102 ff.]{duvaut1976inequalities})
\begin{align}
\partial_{0}(\rho\partial_{0}u)-\Div T & =f\label{eq:dynamic_ela}\\
T & =C\Grad_{c}u,\label{eq:material_law_ela}
\end{align}
where $u\in H_{\nu,0}(\mathbb{R};L_{2}(\Omega)^{3})$ denotes the
displacement field and $T\in H_{\nu,0}(\mathbb{R};L_{2,\mathrm{sym}}(\Omega))$
denotes the stress tensor. Note that due to the domain of the operator
$\Grad_{c}$ we have assumed an implicit boundary condition, which
can be written as 
\[
u=0\mbox{ on }\partial\Omega
\]
in case of a smooth boundary. The function $\rho\in L_{\infty}(\Omega)$
describes the density and is assumed to be real-valued and strictly
positive. The operator $C\in L(L_{2,\mathrm{sym}}(\Omega))$, linking
the stress and the strain tensor $\Grad_{c}u$ is assumed to be selfadjoint
and strictly positive definite. In viscous media it turns out that
the stress $T$ does not only depend on the present state of the strain
tensor, but also on its past. One way to model this relation is to
add a convolution term in \prettyref{eq:material_law_ela}, i.e.,
\begin{equation}
T(t)=C\Grad_{c}u(t)-\intop_{-\infty}^{t}B(t-s)\Grad_{c}u(s)\mbox{ d}s,\quad(t\in\mathbb{R}),\label{eq:material_law_visco}
\end{equation}
where $B\in L_{1,\mu}(\mathbb{R}_{\geq0};L(L_{2,\mathrm{sym}}(\Omega))).$
If we plug \prettyref{eq:material_law_visco} into \prettyref{eq:dynamic_ela}
we end up with the equation, which was considered in \cite{Dafermos1970_asymp_stab}
under the assumption, that $B$ is absolutely continuous.\\
We now show that \prettyref{eq:dynamic_ela} and \prettyref{eq:material_law_visco}
can be written as a system of the form \prettyref{eq:evol}. For this
purpose we define $v\coloneqq\partial_{0}u.$ Note that the operator
$C-B\ast=C^{\frac{1}{2}}\left(1-C^{-\frac{1}{2}}\left(B\ast\right)C^{-\frac{1}{2}}\right)C^{\frac{1}{2}}$
is boundedly invertible, since $C$ is boundedly invertible and $\left(t\mapsto C^{-\frac{1}{2}}B(t)C^{-\frac{1}{2}}\right)\in L_{1,\mu}(\mathbb{R}_{\geq0};L(L_{2,\mathrm{sym}}(\Omega)))$
which gives that $\left(1-C^{-\frac{1}{2}}\left(B\ast\right)C^{-\frac{1}{2}}\right)$
is boundedly invertible on $H_{\nu,0}(\mathbb{R};L_{2,\mathrm{sym}}(\Omega))$
for large $\nu$ (\prettyref{cor:Neumann}). Therefore we can write
\prettyref{eq:material_law_visco} as 
\[
C^{-\frac{1}{2}}\left(1-C^{-\frac{1}{2}}\left(B\ast\right)C^{-\frac{1}{2}}\right)^{-1}C^{-\frac{1}{2}}T=\Grad_{c}u
\]
and by differentiating the last equality we obtain 
\[
\partial_{0}C^{-\frac{1}{2}}\left(1-C^{-\frac{1}{2}}\left(B\ast\right)C^{-\frac{1}{2}}\right)^{-1}C^{-\frac{1}{2}}T=\Grad_{c}v.
\]
Thus, we formally get 
\begin{equation}
\left(\partial_{0}\left(\begin{array}{cc}
\rho & 0\\
0 & C^{-\frac{1}{2}}\left(1-C^{-\frac{1}{2}}\left(B\ast\right)C^{-\frac{1}{2}}\right)^{-1}C^{-\frac{1}{2}}
\end{array}\right)+\left(\begin{array}{cc}
0 & -\Div\\
-\Grad_{c} & 0
\end{array}\right)\right)\left(\begin{array}{c}
v\\
T
\end{array}\right)=\left(\begin{array}{c}
f\\
0
\end{array}\right),\label{eq:visco_ealstic}
\end{equation}
which yields an integro-differential equation of the form \prettyref{eq:evol}
with
\[
M(z)=\left(\begin{array}{cc}
\rho & 0\\
0 & C^{-\frac{1}{2}}\left(1-\sqrt{2\pi}C^{-\frac{1}{2}}\hat{B}(-\i z^{-1})C^{-\frac{1}{2}}\right)^{-1}C^{-\frac{1}{2}}
\end{array}\right)
\]
 and 

\[
A\coloneqq\left(\begin{array}{cc}
0 & -\Div\\
-\Grad_{c} & 0
\end{array}\right),
\]
which is a skew-selfadjoint and hence maximal monotone operator on
the state space $L_{2}(\Omega)^{3}\oplus L_{2,\mathrm{sym}}(\Omega)$.
Thus our solution theory (\prettyref{thm:sol_theory}) applies. So,
in the case that $B$ is absolutely continuous (as it was assumed
in \cite{Dafermos1970_asymp_stab}) we get the well-posedness due
to \prettyref{rem:absolute_cont}. If we do not assume smoothness
for $B$ we end up with the following result.
\begin{thm}
Assume that $B$ satisfies the hypotheses \prettyref{eq:selfadjoint}-\prettyref{eq:imaginary}
and that $C$ and $B(t)$ commute for each $t\in\mathbb{R}.$ Then
\prettyref{eq:visco_ealstic} is well-posed as an equation in $H_{\nu,0}\left(\mathbb{R};L_{2}(\Omega)^{3}\oplus L_{2,\mathrm{sym}}(\Omega)\right)$
for $\nu$ large enough.\end{thm}
\begin{proof}
Note that $C^{-\frac{1}{2}}\left(B\ast\right)C^{-\frac{1}{2}}=\left(C^{-\frac{1}{2}}B(\cdot)C^{-\frac{1}{2}}\right)\ast$.
We show that $C^{-\frac{1}{2}}B(\cdot)C^{-\frac{1}{2}}$ satisfies
the hypotheses \prettyref{eq:selfadjoint}-\prettyref{eq:imaginary}.
The conditions \prettyref{eq:selfadjoint} and \prettyref{eq:commute}
are obvious. Furthermore 
\[
\widehat{\left(C^{-\frac{1}{2}}B(\cdot)C^{-\frac{1}{2}}\right)}(t-\i\nu)=C^{-\frac{1}{2}}\hat{B}(t-\i\nu)C^{-\frac{1}{2}}
\]
for each $t\in\mathbb{R},\nu>\mu$ and hence by the selfadjointness
of $C$ 
\[
\Im\widehat{\left(C^{-\frac{1}{2}}B(\cdot)C^{-\frac{1}{2}}\right)}(t-\i\nu)=C^{-\frac{1}{2}}\Im\hat{B}(t-\i\nu)C^{-\frac{1}{2}}.
\]
This gives 
\begin{align*}
\left\langle \left.t\Im\left(\widehat{C^{-\frac{1}{2}}B(\cdot)C^{-\frac{1}{2}}}\right)(t-\i\nu_{0})\Phi\right|\Phi\right\rangle  & =\left\langle \left.t\Im\hat{B}(t-\i\nu_{0})C^{-\frac{1}{2}}\Phi\right|C^{-\frac{1}{2}}\Phi\right\rangle \\
 & \leq d|C^{-\frac{1}{2}}\Phi|^{2}\\
 & \leq d\|C^{-\frac{1}{2}}\|^{2}|\Phi|^{2}
\end{align*}
for each $\Phi\in L_{2,\mathrm{sym}}(\Omega)$ and $t\in\mathbb{R}.$
Thus, according to \prettyref{prop:pos_B}, the operator $z^{-1}\left(1-\sqrt{2\pi}C^{-\frac{1}{2}}\hat{B}(-\i z^{-1})C^{-\frac{1}{2}}\right)^{-1}$
is uniformly strictly positive definite for $z\in B_{\mathbb{C}}(r_{1},r_{1})$
for sufficiently small $r_{1}>0.$ Therefore, we estimate 
\begin{align*}
 & \Re\left\langle \left.z^{-1}M(z)\left(\begin{array}{c}
v\\
T
\end{array}\right)\right|\left(\begin{array}{c}
v\\
T
\end{array}\right)\right\rangle _{L_{2}(\Omega)^{3}\oplus L_{2,\mathrm{sym}}(\Omega)}\\
 & =\Re\langle z^{-1}\rho v|v\rangle_{L_{2}(\Omega)^{3}}+\Re\langle z^{-1}C^{-\frac{1}{2}}\left(1-\sqrt{2\pi}C^{-\frac{1}{2}}\hat{B}(-\i z^{-1})C^{-\frac{1}{2}}\right)^{-1}C^{-\frac{1}{2}}\Phi|\Phi\rangle_{L_{2,\mathrm{sym}}(\Omega)}\\
 & \geq\Re z^{-1}c_{1}|v|_{L_{2}(\Omega)^{3}}^{2}+c_{2}|C^{-\frac{1}{2}}\Phi|_{L_{2,\mathrm{sym}}(\Omega)}^{2}\\
 & \geq\Re z^{-1}c_{1}|v|_{L_{2}(\Omega)^{3}}^{2}+\frac{c_{2}}{\|C^{\frac{1}{2}}\|^{2}}|\Phi|_{L_{2,\mathrm{sym}}(\Omega)}^{2},
\end{align*}
where we have used $\rho\geq c_{1}>0,\:\Re z^{-1}\left(1-\sqrt{2\pi}C^{-\frac{1}{2}}\hat{B}(-\i z^{-1})C^{-\frac{1}{2}}\right)^{-1}\geq c_{2}>0$. \end{proof}
\begin{rem}
In the theorem above, we have used the solution theory for skew-selfadjoint
$A$ proved in \cite{Picard}. However, if one allows $A$ to be maximal
monotone, this also covers contact problems of visco-elastic solids
with frictional boundary conditions using the theory developed in
\cite{Trostorff2012_nonlin_bd}. Such systems were considered by Migorski
et al. in \cite{Migorski2011} (see also \cite{Rodriguez2007} for
a numerical treatment). 
\end{rem}

\subsection{Materials with memory with non-equilibrium phase transition}

Let $\Omega\subseteq\mathbb{R}^{3}$ be an arbitrary domain. We consider
the following integro-differential inclusion describing the heat conduction
in a two-phase system with long-term memory (see \cite{Colli1993})
\begin{align}
\partial_{0}(1+C\ast)\theta+\partial_{0}(1+D\ast)\chi-\dive(1-K\ast)\grad\theta & =f,\nonumber \\
(\chi,\lambda\theta) & \in\partial_{0}\alpha+\mathbb{L}.\label{eq:phase_trans}
\end{align}
Here $\theta\in H_{\nu,0}(\mathbb{R};L_{2}(\Omega))$ and $\chi\in H_{\nu,0}(\mathbb{R};L_{2}(\Omega))$
are the unknowns representing the absolute temperature and the concentration
of the more energetic phase in our two-phase system, respectively.
The kernels $C$ and $D$ are assumed to lie in $L_{1,\mu}(\mathbb{R}_{\geq0};L(L_{2}(\Omega))$
and $K\in L_{1,\mu}(\mathbb{R}_{\geq0};L(L_{2}(\Omega)^{3}))$ for
some $\mu\geq0.$ The differential inclusion linking $\chi$ and $\theta$
models the phase transition accounting for superheating and supercooling
effects (see \cite{Visintin1985}). Here $\alpha$ and $\lambda$
are positive constants and $\mathbb{L}\subseteq L_{2}(\Omega)\oplus L_{2}(\Omega)$
denotes a maximal monotone relation with $(0,0)\in\mathbb{L}.$ Of
course the equations need to be completed by suitable boundary conditions.
For simplicity we may assume homogeneous Dirichlet boundary conditions
for the temperature $\theta$, i.e. 
\begin{equation}
\theta=0\mbox{ on }\partial\Omega.\label{eq:bd_phase_transition}
\end{equation}
This system (with more advanced boundary conditions) was studied by
Colli et al. (\cite{Colli1993}) in the case of scalar-valued kernels
$C,D$ and $K$ and for the particular maximal monotone relation $\mathbb{L}=H^{-1},$
where $H$ denotes the Heavy-side function. They proved the well-posedness
of the system using a fixed point argument and studied the limit problem
as $\alpha$ tends to 0 (the so-called Stefan problem, see e.g. \cite{Visintin1985}).
We will show that our approach allows to relax the assumptions on
the kernels and allows that $\mathbb{L}$ is an arbitrary maximal
monotone relation. For doing so, we begin to define the differential
operators involved.
\begin{defn}
We define the operator $\grad_{c}$ as the closure of 
\begin{align*}
\grad|_{C_{c}^{\infty}(\Omega)}:C_{c}^{\infty}(\Omega)\subseteq L_{2}(\Omega) & \to L_{2}(\Omega)^{3}\\
\phi & \mapsto(\partial_{1}\phi,\partial_{2}\phi,\partial_{3}\phi)^{T}
\end{align*}
and $\dive\coloneqq-(\grad_{c})^{\ast}.$\end{defn}
\begin{rem}
The domain of $\grad_{c}$ is the classical Sobolev-space $H_{0}^{1}(\Omega)$
of weakly differentiable $L_{2}$-functions, satisfying an abstract
homogeneous Dirichlet boundary condition. The domain of $\dive$ is
the maximal domain of the divergence on $L_{2}(\Omega),$ i.e. the
set of $L_{2}$-vector fields, whose distributional divergence is
again an $L_{2}$-function.
\end{rem}
In order to rewrite the system \prettyref{eq:phase_trans} as an evolutionary
inclusion, we integrate the first equation of \prettyref{eq:phase_trans},
i.e. we apply the operator $\partial_{0}^{-1}.$ Thus, we end up with
\begin{align*}
(1+C\ast)\theta+(1+D\ast)\chi-\dive\partial_{0}^{-1}(1-K\ast)\grad\theta & =\partial_{0}^{-1}f,\\
(\chi,\lambda\theta) & \in\partial_{0}\alpha+\mathbb{L}.
\end{align*}
Defining $q\coloneqq-\partial_{0}^{-1}(1-K\ast)\grad\theta,$ the
latter system, subject to the boundary condition \prettyref{eq:bd_phase_transition},
can be written as 
\begin{equation}
\left(\left(\begin{array}{c}
\theta\\
\chi\\
q
\end{array}\right),\left(\begin{array}{c}
\partial_{0}^{-1}f\\
0\\
0
\end{array}\right)\right)\in\partial_{0}\left(\begin{array}{ccc}
0 & 0 & 0\\
0 & \alpha & 0\\
0 & 0 & (1-K\ast)^{-1}
\end{array}\right)+\left(\begin{array}{ccc}
1+C\ast & 1+D\ast & 0\\
-\lambda & 0 & 0\\
0 & 0 & 0
\end{array}\right)+\left(\begin{array}{ccc}
0 & 0 & \dive\\
0 & \mathbb{L} & 0\\
\grad_{c} & 0 & 0
\end{array}\right),\label{eq:phase_trans_evo}
\end{equation}
in $H_{\nu,0}(\mathbb{R};L_{2}(\Omega)\oplus L_{2}(\Omega)\oplus L_{2}(\Omega)^{3})$,
where we have chosen $\nu>0$ large enough, such that $1-K\ast$ gets
boundedly invertible (compare \prettyref{cor:Neumann}). We note that
\[
A\coloneqq\left(\begin{array}{ccc}
0 & 0 & \dive\\
0 & \mathbb{L} & 0\\
\grad_{c} & 0 & 0
\end{array}\right)
\]
is maximal monotone with $(0,0)\in A$, since $\left(\begin{array}{ccc}
0 & 0 & \dive\\
0 & 0 & 0\\
\grad_{c} & 0 & 0
\end{array}\right)$ is skew-selfadjoint and $\mathbb{L}$ is maximal monotone with $(0,0)\in\mathbb{L}.$
The material law is of the form 
\[
M(z)=\left(\begin{array}{ccc}
0 & 0 & 0\\
0 & \alpha & 0\\
0 & 0 & (1-\sqrt{2\pi}\hat{K}(-\i z^{-1}))^{-1}
\end{array}\right)+z\left(\begin{array}{ccc}
1+\sqrt{2\pi}\hat{C}(-\i z^{-1}) & 1+\sqrt{2\pi}\hat{D}(-\i z^{-1}) & 0\\
-\lambda & 0 & 0\\
0 & 0 & 0
\end{array}\right).
\]

\begin{thm}
Assume that $K$ satisfies the hypotheses \prettyref{eq:selfadjoint}-\prettyref{eq:imaginary}.
Then the integro-differential inclusion \prettyref{eq:phase_trans_evo}
is well-posed. \end{thm}
\begin{proof}
It is left to show that $M$ satisfies the solvability condition \prettyref{eq:solv}.
According to \prettyref{prop:pos_B} we have 
\[
\Re z^{-1}(1-\sqrt{2\pi}\hat{K}(-\i z^{-1}))^{-1}\geq c_{0}>0
\]
for some $c_{0}>0$ and all $z\in B_{\mathbb{C}}(r_{1},r_{1}),$ where
$r_{1}>0$ is chosen sufficiently small. Thus, it is left to estimate
\[
N(z)\coloneqq\left(\begin{array}{cc}
1+\sqrt{2\pi}\hat{C}(-\i z^{-1}) & 1+\sqrt{2\pi}\hat{D}(-\i z^{-1})\\
-\lambda & z^{-1}\alpha
\end{array}\right)
\]
for $z\in B_{\mathbb{C}}(r_{2},r_{2}),$ for suitable $0<r_{2}\leq r_{1}.$
For $u\coloneqq(\theta,\chi)\in L_{2}(\Omega)\oplus L_{2}(\Omega)$
and $z^{-1}=\i t+\nu$ for some $t\in\mathbb{R},\nu>\frac{1}{2r_{2}}$
we estimate 
\begin{align*}
 & \Re\langle N(z)u|u\rangle\\
 & =\Re\langle\left(1+\sqrt{2\pi}\hat{C}(-\i z^{-1})\right)\theta|\theta\rangle+\Re\langle\left(1+\sqrt{2\pi}\hat{D}(-\i z^{-1})\right)\chi|\theta\rangle-\Re\langle\lambda\theta|\chi\rangle+\Re\langle z^{-1}\alpha\chi|\chi\rangle\\
 & \geq(1-|C|_{1,\nu})|\theta|^{2}-(1+|D|_{1,\nu}+\lambda)|\chi||\theta|+\nu\alpha|\chi|^{2}\\
 & \geq\left(1-|C|_{1,\nu}-\frac{\varepsilon^{2}}{2}\right)|\theta|^{2}+\left(\nu\alpha-\frac{\left(1+|D|_{1,\nu}+\lambda\right)^{2}}{2\varepsilon^{2}}\right)|\chi|^{2}
\end{align*}
for all $\varepsilon>0.$ Choosing $r_{2}$ and $\varepsilon$ small
enough (and hence $\nu$ great enough) the latter can be estimated
by 
\[
\Re\langle N(z)u|u\rangle\geq c_{1}|u|^{2}\quad(z\in B_{\mathbb{C}}(r_{2},r_{2}))
\]
for some $c_{1}>0$. Thus, $M$ satisfies \prettyref{eq:solv} and
so \prettyref{thm:sol_theory} applies. \end{proof}
\begin{rem}
We note that in contrast to \cite{Colli1993} we do not need to impose
any regularity assumption on the kernel $C$. 
\end{rem}


\begin{thebibliography}{10}

\bibitem{Aizicovici1980}
S.~Aizicovici.
\newblock {On a semilinear Volterra integrodifferential equation.}
\newblock {\em Isr. J. Math.}, 36:273--284, 1980.

\bibitem{Berrimi2006}
S.~Berrimi and S.~A. Messaoudi.
\newblock {Existence and decay of solutions of a viscoelastic equation with a
  nonlinear source.}
\newblock {\em Nonlinear Anal., Theory Methods Appl.}, 64(10):2314--2331, 2006.

\bibitem{Brezis}
H.~Brezis.
\newblock {\em Operateurs maximaux monotones et semi-groupes de contractions
  dans les espaces de Hilbert}.
\newblock Universite Paris VI et CNRS, 1971.

\bibitem{Cannarsa2003}
P.~Cannarsa and D.~Sforza.
\newblock {Global solutions of abstract semilinear parabolic equations with
  memory terms.}
\newblock {\em NoDEA, Nonlinear Differ. Equ. Appl.}, 10(4):399--430, 2003.

\bibitem{Cannarsa2011}
P.~Cannarsa and D.~Sforza.
\newblock {Integro-differential equations of hyperbolic type with positive
  definite kernels.}
\newblock {\em J. Differ. Equations}, 250(12):4289--4335, 2011.

\bibitem{Cavalcanti2003}
M.~M. Cavalcanti and H.~P. Oquendo.
\newblock {Frictional versus viscoelastic damping in a semilinear wave
  equation.}
\newblock {\em SIAM J. Control Optimization}, 42(4):1310--1324, 2003.

\bibitem{Clement1992}
P.~Cl{\'e}ment and J.~Pr{\"u}ss.
\newblock {Global existence for a semilinear parabolic Volterra equation.}
\newblock {\em Math. Z.}, 209(1):17--26, 1992.

\bibitem{Colli1993}
P.~Colli and M.~Grasselli.
\newblock {Phase transition problems in materials with memory.}
\newblock {\em J. Integral Equations Appl.}, 5(1):1--22, 1993.

\bibitem{Dafermos1970_abtract_Volterra}
C.~Dafermos.
\newblock {An abstract Volterra equation with applications to linear
  viscoelasticity.}
\newblock {\em J. Differ. Equations}, 7:554--569, 1970.

\bibitem{Dafermos1970_asymp_stab}
C.~Dafermos.
\newblock {Asymptotic stability in viscoelasticity.}
\newblock {\em Arch. Ration. Mech. Anal.}, 37:297--308, 1970.

\bibitem{diekmann1995delay}
O.~Diekmann, S.~{van Gils}, S.~Lunel, and H.~Walther.
\newblock {\em Delay Equations: Functional-, Complex-, and Nonlinear Analysis}.
\newblock Number Bd. 110 in Applied Mathematical Sciences. Springer, 1995.

\bibitem{duvaut1976inequalities}
G.~Duvaut and J.~L. Lions.
\newblock {\em {Inequalities in mechanics and physics}}.
\newblock Grundlehren der mathematischen Wissenschaften. Springer-Verlag, 1976.

\bibitem{fabrizio1987mathematical}
M.~Fabrizio and A.~Morro.
\newblock {\em Mathematical Problems in Linear Viscoelasticity}.
\newblock Siam Studies in Applied Mathematics. Society for Industrial and
  Applied Mathematics, 1987.

\bibitem{Gripenberg1990}
G.~Gripenberg, S.-O. Londen, and O.~Staffans.
\newblock {\em {Volterra integral and functional equations.}}
\newblock {Encyclopedia of Mathematics and Its Applications, 34. Cambridge
  etc.: Cambridge University Press. xxii, 701 p.}, 1990.

\bibitem{hale1993introduction}
J.~Hale and S.~Lunel.
\newblock {\em Introduction to Functional Differential Equations}.
\newblock Applied Mathematical Sciences. Springer, 1993.

\bibitem{hille1957functional}
E.~Hille and R.~S. Phillips.
\newblock {\em {Functional analysis and semi-groups}}.
\newblock Colloquium Publications - American Mathematical Society. American
  Mathematical Society, 1957.

\bibitem{hu2000handbook}
S.~Hu and N.~S. Papageorgiou.
\newblock {\em {Handbook of Multivalued Analysis}}, volume 2: Applications, of
  {\em Mathematics and its applications}.
\newblock Kluwer Academic Publishers, 2000.

\bibitem{Kalauch2011}
A.~{Kalauch}, R.~{Picard}, S.~{Siegmund}, S.~{Trostorff}, and M.~{Waurick}.
\newblock {A Hilbert Space Perspective on Ordinary Differential Equations with
  Memory Term}.
\newblock Technical report, TU Dresden, 2011.
\newblock arXiv:1204.2924, to appear in J. Dyn. Differ.

\bibitem{Migorski2011}
S.~Mig{\'o}rski, A.~Ochal, and M.~Sofonea.
\newblock {Analysis of a frictional contact problem for viscoelastic materials
  with long memory.}
\newblock {\em Discrete Contin. Dyn. Syst., Ser. B}, 15(3):687--705, 2011.

\bibitem{Nunziato1971}
J.~W. Nunziato.
\newblock {On heat conduction in materials with memory.}
\newblock {\em Q. Appl. Math.}, 29:187--204, 1971.

\bibitem{Picard2000}
R.~Picard.
\newblock {Evolution equations as operator equations in lattices of Hilbert
  spaces.}
\newblock {\em Glas. Mat., III. Ser.}, 35(1):111--136, 2000.

\bibitem{Picard}
R.~Picard.
\newblock {A structural observation for linear material laws in classical
  mathematical physics.}
\newblock {\em Math. Methods Appl. Sci.}, 32(14):1768--1803, 2009.

\bibitem{Picard2010}
R.~Picard.
\newblock {On a comprehensive class of linear material laws in classical
  mathematical physics.}
\newblock {\em Discrete Contin. Dyn. Syst., Ser. S}, 3(2):339--349, 2010.

\bibitem{Picard_McGhee}
R.~Picard and D.~McGhee.
\newblock {\em {Partial differential equations. A unified Hilbert space
  approach.}}
\newblock {de Gruyter Expositions in Mathematics 55. Berlin: de Gruyter.
  xviii}, 2011.

\bibitem{pruss1993evolutionary}
J.~Pr{\"u}ss.
\newblock {\em Evolutionary integral equations and applications}.
\newblock Monographs in mathematics. Birkh{\"a}user Verlag, 1993.

\bibitem{Pruss2009}
J.~Pr{\"u}ss.
\newblock {Decay properties for the solutions of a partial differential
  equation with memory.}
\newblock {\em Arch. Math.}, 92(2):158--173, 2009.

\bibitem{Rodriguez2007}
{\'A}.~Rodr{\'i}guez-Ar{\'o}s, J.~Via{\~n}o, and M.~Sofonea.
\newblock Numerical analysis of a frictional contact problem for viscoelastic
  materials with long-term memory.
\newblock {\em Numerische Mathematik}, 108(2):327--358, 2007.

\bibitem{rudin1987real}
W.~Rudin.
\newblock {\em {Real and complex analysis}}.
\newblock Mathematics series. McGraw-Hill, 1987.

\bibitem{showalter_book}
R.~E. Showalter.
\newblock {\em Monotone Operators in Banach Space and Nonlinear Partial
  Differential Equations}.
\newblock American Mathematical Society, 1997.

\bibitem{stein2003fourier}
E.~Stein and R.~Shakarchi.
\newblock {\em Fourier Analysis: An Introduction}.
\newblock Number Bd. 10 in Princeton Lectures in Analysis. Princeton University
  Press, 2003.

\bibitem{Trostorff_2011}
S.~Trostorff.
\newblock {\em Well-posedness and causality for a class of evolutionary
  inclusions}.
\newblock PhD thesis, TU Dresden, 2011.
\newblock URL:
  \url{http://www.qucosa.de/fileadmin/data/qucosa/documents/7832/phd-thesis_trostorff.pdf}.

\bibitem{Trostorff2012_NA}
S.~Trostorff.
\newblock {An alternative approach to well-posedness of a class of differential
  inclusions in Hilbert spaces.}
\newblock {\em Nonlinear Anal., Theory Methods Appl., Ser. A, Theory Methods},
  75(15):5851--5865, 2012.

\bibitem{Trostorff2012_nonlin_bd}
S.~Trostorff.
\newblock {Autonomous Evolutionary Inclusions with Applications to Problems
  with Nonlinear Boundary Conditions.}
\newblock {\em Int. J. Pure Appl. Math.}, 85(2):303--338, 2013.

\bibitem{Visintin1985}
A.~Visintin.
\newblock {Supercooling and superheating effects in phase transitions.}
\newblock {\em IMA J. Appl. Math.}, 35:223--256, 1985.

\bibitem{Vlasov2010}
V.~Vlasov and D.~A. Medvedev.
\newblock {Functional-differential equations in Sobolev spaces and related
  problems in spectral theroy.}
\newblock {\em J. Math. Sci., New York}, 164(5):659--842, 2010.

\bibitem{Yosida}
K.~Yosida.
\newblock {\em Functional Analysis}.
\newblock Springer, 6th edition, 1995.

\end{thebibliography}
\end{document}